\documentclass[12pt,lenq]{amsart}
\usepackage{amsmath}
\usepackage{amscd}
\usepackage{amssymb}
\usepackage{amsbsy}
\usepackage{amsfonts}
\usepackage{latexsym}
\usepackage{graphics}
\usepackage{graphicx}
\usepackage{amsmath,amscd,latexsym}
\usepackage{multirow}
\usepackage{array}
\usepackage{paralist}
\usepackage{titletoc}

\theoremstyle{plain}
\newtheorem{theorem}{Theorem}[section]
\newtheorem{proposition}[theorem]{Proposition}
\newtheorem{lemma}[theorem]{Lemma}
\newtheorem{corollary}[theorem]{Corollary}
\newtheorem{main theorem}[theorem]{Main Theorem}

\newtheorem{question}{Question}
\newtheorem{remark}{Remark}
\newtheorem{definition}{Definition}
\newtheorem{convention}{Convention}
\newtheorem{example}{Example}

\newlength\savewidth

\newcommand{\interior}{\operatorname{int}}

\newcommand{\ZZ}{\mathbb{Z}}
\newcommand{\QQ}{\mathbb{Q}}
\newcommand{\RR}{\mathbb{R}}

\newcommand{\HH}{\mathbb{H}}
\newcommand{\QQQ}{\hat{\mathbb{Q}}}
\newcommand{\RRR}{\hat{\mathbb{R}}}

\newcommand{\Conway}{\mbox{\boldmath$S$}^{2}}
\newcommand{\Conways}
{(\mbox{\boldmath$S$}^{2},\mbox{\boldmath$P$})}
\newcommand{\PP}{\mbox{\boldmath$P$}}
\newcommand{\PConway}{\mbox{\boldmath$S$}}

\newcommand{\rtangle}[1]{(B^3,t({#1}))}

\newcommand{\DD}{\mathcal{D}}
\newcommand{\RGPP}[1]{\hat\Gamma_{#1}}
\newcommand{\RGP}[1]{\Gamma_{#1}}

\newcommand{\llangle}{\langle\langle}
\newcommand{\rrangle}{\rangle\rangle}

\newcommand{\lp}{(\hskip -0.07cm (}
\newcommand{\rp}{)\hskip -0.07cm )}

\begin{document}

\title[Epimorphisms between 2-bridge link groups:
Homotopically trivial simple loops on 2-bridge spheres]
{Epimorphisms between 2-bridge link groups:
Homotopically trivial simple loops on 2-bridge spheres}
\author{Donghi Lee}
\address{Department of Mathematics\\
Pusan National University \\
San-30 Jangjeon-Dong, Geumjung-Gu, Pusan, 609-735, Korea}
\email{donghi@pusan.ac.kr}

\author{Makoto Sakuma}
\address{Department of Mathematics\\
Graduate School of Science\\
Hiroshima University\\
Higashi-Hiroshima, 739-8526, Japan}
\email{sakuma@math.sci.hiroshima-u.ac.jp}

\subjclass[2010]{Primary 57M25, 20F06 \\
\indent {The first author was supported by Basic Science Research Program
through the National Research Foundation of Korea(NRF) funded
by the Ministry of Education, Science and Technology(2009-0065798).
The second author was supported
by JSPS Grants-in-Aid 18340018 and 21654011,
and was partially supported by JSPS Core-to-Core Program 18005.}}


\begin{abstract} We give a complete characterization of those
essential simple loops on $2$-bridge spheres of $2$-bridge links
which are null-homotopic in the link complements.
By using this result,
we describe all upper-meridian-pair-preserving epimorphisms between 2-bridge link groups.
\end{abstract}
\maketitle

\tableofcontents

\section{Introduction}

For a knot or a link, $K$, in $S^3$, the fundamental group $\pi_1(S^3-K)$ of the complement is called
the {\it knot group} or the {\it link group} of $K$,
and is denoted by $G(K)$.
For prime knots, the knot groups are complete invariants
for the knot types
(see \cite{Gordon-Luecke}).
Moreover we have a partial order on the set of prime knots,
by setting $\tilde K\ge K$ if there is an
epimorphism $G(\tilde{K}) \to G(K)$
(see, for example, \cite[Proposition ~3.2]{Silver-Whitten}).
Epimorphisms among link groups have received
considerable attention
and they have been studied in various places in the literature
(see
\cite{Boileau-Boyer-Reid-Wang, Boileau-Rubinstein-Wang,
Gonzalez-Ramirez1, Gonzalez-Ramirez2,
HKMS, Hoste-Shanahan,
Kitano-Suzuki, Kitano-Suzuki2, Kitano-Suzuki-Wada,
Ohtsuki-Riley-Sakuma, Reid-Wang-Zhou,
Sakuma2, Silver-Whitten, Simon, Soma1, Soma2, Wang}
and references therein).

In \cite{Ohtsuki-Riley-Sakuma},
a systematic construction of epimorphisms
between 2-bridge link groups was given.
The construction is based on a systematic construction of
essential simple loops on $2$-bridge spheres of $2$-bridge links
which are null-homotopic in the link complements.
Thus the following question naturally arises
(see \cite[Question ~9.1(2)]{Ohtsuki-Riley-Sakuma}).

\begin{question} \label{question}
Let $K$ be a 2-bridge link, and let $S$ be
a $4$-times punctured sphere in $S^3-K$
determined by a $2$-bridge sphere.
Then which essential simple loops on $S$ are null-homotopic in
$S^3-K$?
\end{question}

It should be noted that each $2$-bridge link admits a unique
$2$-bridge sphere up to isotopy (see \cite{Schubert}),
and hence the $4$-times punctured sphere $S$ in the above problem
is unique up to isotopy.

In this paper, we give a complete answer to the above question
(Main Theorem ~\ref{main_theorem}).
In fact, we show that those essential simple loops on $S$
constructed in \cite[Corollary ~4.7]{Ohtsuki-Riley-Sakuma}
are the only essential simple loops on $S$
which are null-homotopic in the $2$-bridge link complement.
This enables us to describe
all epimorphisms between $2$-bridge link groups
which map the upper meridian pair of the source group
to the upper meridian pair of the target group
(Main Theorem ~\ref{main_theorem_2}).
In fact, this theorem says that any such epimorphism is
equivalent to that constructed in
\cite[Theorem ~1.1]{Ohtsuki-Riley-Sakuma}.

To the authors' knowledge,
every known pair of $2$-bridge knots $(\tilde K, K)$ with $\tilde K\ge K$
belongs to the list in \cite[Theorem ~1.1]{Ohtsuki-Riley-Sakuma}.
Kitano and Suzuki and their coworkers
verified this for $2$-bridge knots up to 11-crossings
in \cite{HKMS, Kitano-Suzuki, Kitano-Suzuki2}.
Gonzal\'ez-Ac\~una and Ram\'irez \cite{Gonzalez-Ramirez2}
determined the $2$-bridge knots whose knot groups have
epimorphisms to the $(2,p)$ torus knot group,
and their result implies that every such $2$-bridge knot group
is isomorphic to one
constructed in \cite[Theorem ~1.1]{Ohtsuki-Riley-Sakuma}.
In their recent work \cite{Boileau-Boyer-Reid-Wang},
Boileau, Boyer, Reid and Wang proved Simon's conjecture
(see \cite[Problem ~1.12(D)]{Kirby}) for $2$-bridge knot groups,
namely they have shown that each $2$-bridge knot group surjects onto
only finitely many distinct knot groups.
To be more precise, they have shown that
if a $2$-bridge knot group $G(K)$ surjects onto a non-trivial
knot group $G(K')$, then $K'$ is a $2$-bridge knot
and the epimorphism is induced by a map between the knot complements
of non-zero degree.
The last condition is
satisfied for all epimorphisms
in \cite[Theorem ~1.1]{Ohtsuki-Riley-Sakuma}.
In fact, they are induced by a
very nice map $(S^3,K)\to (S^3,K')$,
called a branched-fold map \cite[Theorem ~1.2]{Ohtsuki-Riley-Sakuma}.
Thus it would be natural to expect that
any epimorphism between $2$-bridge knot groups
is equivalent to one in \cite[Theorem ~1.1]{Ohtsuki-Riley-Sakuma}.
In fact, some evidence for this conjecture
was provided recently by Hoste and Shanahan ~\cite{Hoste-Shanahan}.

Question ~\ref{question} can be regarded as a
special case of the more general question
that, for a given link $L$ and a bridge sphere $F$ for $L$,
which essential simple loops on $F$ are null-homotopic in $S^3-L$.
The latter question in turn can be regarded as a variation
of the question that, for a given $3$-manifold $M$ and its
Heegaard surface $F$,
which essential simple loops on $F$ are null-homotopic in $M$.
In \cite[Question ~5.4]{Gordon}, Minsky
refined this to a certain question
which generalizes Question ~\ref{question}.
Thus our result
may be regarded as
an answer to a special variation of Minsky's question
(see Section ~\ref{Further_discussion}).

The authors would like to thank
Norbert A'Campo, Hirotaka Akiyoshi,
Brian Bowditch, Danny Calegari,
Max Forester,
Koji Fujiwara,
Yair Minsky, Ser Peow Tan
and Caroline Series for stimulating conversations.
They also thank the referee for his/her careful reading
of the manuscript.

\section{Main result} \label{statements}

Consider the discrete group, $H$, of isometries
of the Euclidean plane $\RR^2$
generated by the $\pi$-rotations around
the points in the lattice $\ZZ^2$.
Set $\Conways:=(\RR^2,\ZZ^2)/H$
and call it the {\it Conway sphere}.
Then $\Conway$ is homeomorphic to the 2-sphere,
and $\PP$ consists of four points in $\Conway$.
We also call $\Conway$ the Conway sphere.
Let $\PConway:=\Conway-\PP$ be the complementary
4-times punctured sphere.
For each $r \in \QQQ:=\QQ\cup\{\infty\}$,
let $\alpha_r$ be the simple loop in $\PConway$
obtained as the projection of a line in $\RR^2-\ZZ^2$
of slope $r$.
Then $\alpha_r$ is {\it essential} in $\PConway$,
i.e., it does not bound a disk in $\PConway$
and is not homotopic to a loop around a puncture.
Conversely, any essential simple loop in $\PConway$
is isotopic to $\alpha_r$ for
a unique $r\in\QQQ$.
Then $r$ is called the {\it slope} of the simple loop.
Similarly, any simple arc $\delta$
in $\Conway$ joining two
different points in $\PP$ such that
$\delta\cap \PP=\partial\delta$
is isotopic to the image of a line in $\RR^2$
of some slope $r\in\QQQ$
which intersects $\ZZ^2$.
We call $r$ the {\it slope} of $\delta$.

A {\it trivial tangle} is a pair $(B^3,t)$,
where $B^3$ is a 3-ball and $t$ is a union of two
arcs properly embedded in $B^3$
which is parallel to a union of two
mutually disjoint arcs in $\partial B^3$.
Let $\tau$ be the simple unknotted arc in $B^3$
joining  the two components of $t$
as illustrated in Figure ~\ref{fig.trivial-tangle}.
We call it the {\it core tunnel} of the trivial tangle.
Pick a base point $x_0$ in $\interior \tau$,
and let $(\mu_1,\mu_2)$ be the generating pair
of the fundamental group $\pi_1(B^3-t,x_0)$
each of which is represented by
a based loop consisting of
a small peripheral simple loop
around a component of $t$ and
a subarc of $\tau$ joining the circle to $x_0$.
For any base point $x\in B^3-t$,
the generating pair of $\pi_1(B^3-t,x)$
corresponding to the generating pair $(\mu_1,\mu_2)$
of $\pi_1(B^3-t,x_0)$ via a path joining $x$ to $x_0$
is denoted by the same symbol.
The pair $(\mu_1,\mu_2)$ is unique up to
(i) reversal of the order,
(ii) replacement of one of the members with its inverse,
and (iii) simultaneous conjugation.
We call the equivalence class of $(\mu_1,\mu_2)$
the {\it meridian pair} of the fundamental
group $\pi_1(B^3-t)$.

\begin{figure}[h]
\begin{center}
\includegraphics{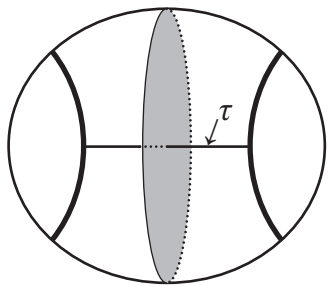}
\end{center}
\caption{\label{fig.trivial-tangle}
A trivial tangle}
\end{figure}

By a {\it rational tangle},
we mean a trivial tangle $(B^3,t)$
which is endowed with a homeomorphism from
$\partial(B^3,t)$ to $\Conways$.
Through the homeomorphism we identify
the boundary of a rational tangle with the Conway sphere.
Thus the slope of an essential simple loop in
$\partial B^3-t$ is defined.
We define
the {\it slope} of a rational tangle
to be the slope of
an essential loop on $\partial B^3 -t$ which bounds a disk in $B^3$
separating the components of $t$.
(Such a loop is unique up to isotopy
on $\partial B^3 -t$ and is called a {\it meridian}
of the rational tangle.)
We denote a rational tangle of slope $r$ by
$\rtangle{r}$.
By van Kampen's theorem, the fundamental group
$\pi_1(B^3-t(r))$ is identified
with the quotient
$\pi_1(\PConway)/\llangle\alpha_r \rrangle$,
where $\llangle\alpha_r \rrangle$ denotes the normal
closure.

For each $r\in \QQQ$,
the {\it 2-bridge link $K(r)$ of slope $r$}
is defined to be the sum of the rational tangles of slopes
$\infty$ and $r$, namely,
$(S^3,K(r))$ is
obtained from $\rtangle{\infty}$ and
$\rtangle{r}$
by identifying their boundaries through the
identity map on the Conway sphere
$\Conways$. (Recall that the boundaries of
rational tangles are identified with the Conway sphere.)
$K(r)$ has one or two components according as
the denominator of $r$ is odd or even.
We call $\rtangle{\infty}$
and $\rtangle{r}$, respectively,
the {\it upper tangle} and {\it lower tangle}
of the 2-bridge link.
The 2-bridge links are classified by the following theorem
of Schubert \cite{Schubert}
(cf. \cite{Burde-Zieschang, Kawauchi}).

\begin{theorem}[Schubert]
\label{Thm:Schubert}
Two 2-bridge links $K(q/p)$ and $K(q'/p')$
are equivalent
(i.e., there is a homeomorphism from $S^3$ to itself
sending $K(q/p)$ to $K(q'/p')$),
if and only if the following conditions hold.
\begin{enumerate}[\indent \rm (1)]
\item
$p=p'$.
\item
Either $q\equiv \pm q'\pmod p$
or $qq'\equiv \pm 1\pmod p$.
\end{enumerate}
\end{theorem}

Let $\DD$ be the
{\it Farey tessellation},
that is,
the tessellation of the upper half
space $\HH^2$ by ideal triangles which are obtained
from the ideal triangle with the ideal vertices $0, 1,
\infty \in \QQQ$ by repeated reflection in the edges.
Then $\QQQ$ is identified with the set of the ideal vertices of $\DD$.
For each $r\in \QQQ$,
let $\RGP{r}$ be the group of automorphisms of
$\DD$ generated by reflections in the edges of $\DD$
with an endpoint $r$.
It should be noted that $\RGP{r}$
is isomorphic to the infinite dihedral group
and the region bounded by two adjacent edges of $\DD$
with an endpoint $r$ is a fundamental domain
for the action of $\RGP{r}$ on $\HH^2$,
by virtue of Poincare's fundamental polyhedron theorem
(see, for example, \cite{Ratcliffe}).
Let $\RGPP{r}$ be the group generated by $\RGP{r}$ and $\RGP{\infty}$.
When $r\in \QQ - \ZZ$,
$\RGPP{r}$ is equal to the free product $\RGP{r}*\RGP{\infty}$,
having a fundamental domain shown in Figure ~\ref{fig.fd}.
Otherwise, $\RGPP{r}$ is the group generated by
reflections in
the edges of $\DD$ or $\RGP{\infty}$ according as $r\in\ZZ$ or $r=\infty$.
It should be noted that Theorem ~\ref{Thm:Schubert}
says that two 2-bridge links $K(r)$ and $K(r')$ are equivalent
if and only if there is an automorphism of $\DD$
which sends $\{\infty,r\}$ to $\{\infty,r'\}$.
Thus the conjugacy class of the group $\RGPP{r}$
in the automorphism group of $\DD$ is uniquely determined by
the link $K(r)$.

\begin{figure}[h]
\begin{center}
\includegraphics{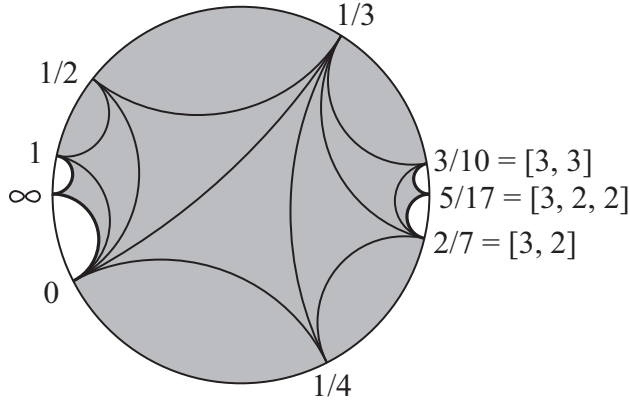}
\end{center}
\caption{\label{fig.fd}
A fundamental domain of $\hat\Gamma_r$ in the
Farey tessellation (the shaded domain) for $r=5/17=${\scriptsize $\cfrac{1}{3+\cfrac{1}{2+\cfrac{1}{2}}}$}\,$=:[3,2,2]$.}
\end{figure}

We recall the following fact
(\cite[Proposition ~4.6 and Corollary ~4.7]{Ohtsuki-Riley-Sakuma})
which describes
the role of $\RGPP{r}$ in the study of
$2$-bridge link groups.

\begin{proposition}\label{Prop:Knotgroup}
For every 2-bridge link $K(r)$, the following holds.
If two elements $s$ and $s'$ of $\QQQ$
lie in the same $\RGPP{r}$-orbit,
then $\alpha_s$ and $\alpha_{s'}$ are homotopic in
$S^3-K(r)$.
In particular, if $s$ belongs to the orbit of $\infty$ or $r$
by $\RGPP{r}$,
then $\alpha_s$ is null-homotopic in $S^3-K(r)$.
\end{proposition}

Our main theorem says that the converse to the
last statement in the above proposition is valid.

\begin{main theorem}\label{main_theorem}
The loop $\alpha_s$ is null-homotopic in $S^3 - K(r)$
if and only if $s$ belongs to the $\RGPP{r}$-orbit of $\infty$ or $r$.
\end{main theorem}

This theorem may be paraphrased as follows,
with a detailed reason explained in Section ~\ref{group_presentation}.

\begin{main theorem} \label{main_theorem_2}
There is an upper-meridian-pair-preserving epimorphism
from $G(K(s))$ to $G(K(r))$
if and only if
$s$ or $s+1$
belongs to the $\RGPP{r}$-orbit of $r$ or $\infty$.
\end{main theorem}

Since the if part is \cite[Theorem ~1.1]{Ohtsuki-Riley-Sakuma},
the heart of this theorem is the only if part.

The remainder of this paper is organized as follows.
In Section ~\ref{group_presentation},
we introduce the so-called upper presentation
$G(K(r))=\langle a, b \, | \, u_r \rangle$ of a 2-bridge link group,
where $\{a, b\}$ is the upper meridian pair of $K(r)$. This upper presentation of
a 2-bridge link group will be used throughout this paper.
In Section ~\ref{sequences},
we define two sequences $S(r)$ and $T(r)$ of slope $r$
and two cyclic sequences $CS(r)$ and $CT(r)$ of slope $r$
all of which arise from the single relator $u_r$ of the presentation
$G(K(r))=\langle a, b \, | \, u_r \rangle$,
and observe several important properties of these sequences
so that we can adopt, in the succeeding sections,
small cancellation theory which is one of the geometric techniques
in combinatorial group theory.
In Section ~\ref{small_cancellation},
we show that the presentation
$G(K(r))=\langle a, b \, | \, u_r \rangle$,
where $0<r<1$,
satisfies small cancellation conditions $C(4)$ and $T(4)$.
In Section ~\ref{van_Kampen_diagrams},
by applying the Curvature Formula of Lyndon and Schupp (see \cite{lyndon_schupp})
to van Kampen diagrams over $G(K(r))=\langle a, b \, | \, u_r \rangle$,
we obtain that if $\alpha_s$ is null-homotopic in $S^3-K(r)$, where $0<r<1$,
then the cyclic word $(u_s)$ contains
some particular part of the the cyclic word $(u_r^{\pm 1})$.
In Section ~\ref{Proof_of_MainTheorem},
we prove the only if part of Main Theorem ~\ref{main_theorem}
by showing that if a rational number $s$ belongs to
a natural fundamental domain of the action of $\RGPP{r}$
on the domain of discontinuity of $\RGPP{r}$,
then $\alpha_s$ is not null-homotopic in $S^3-K(r)$.
In the final section, Section ~\ref{Further_discussion},
we describe the relation of
Main Theorem ~\ref{main_theorem} with
the question raised by Minsky in \cite[Question ~5.4]{Gordon}.

\section{Presentations of 2-bridge link groups} \label{group_presentation}

In this section, we introduce the
upper presentation of a 2-bridge link group which we shall use throughout this paper.
By van Kampen's theorem, the link group $G(K(r))=\pi_1(S^3-K(r))$ is identified with
$\pi_1(\PConway)/ \llangle\alpha_{\infty},\alpha_r\rrangle$. We call the image in the link group
of the meridian pair of the fundamental group $\pi_1(B^3-t(\infty))$ (resp. $\pi_1(B^3-t(r))$
the {\it upper meridian pair} (resp. {\it lower meridian pair}). The link group is regarded as
the quotient of the rank 2 free group, $\pi_1(B^3-t(\infty))\cong\pi_1(\PConway)/
\llangle\alpha_{\infty}\rrangle$, by the normal closure of $\alpha_{r}$. This gives a one-relator
presentation of the link group, which is
called the {\it upper presentation} (see \cite{Crowell-Fox}).

\begin{figure}[h]
\begin{center}
\includegraphics{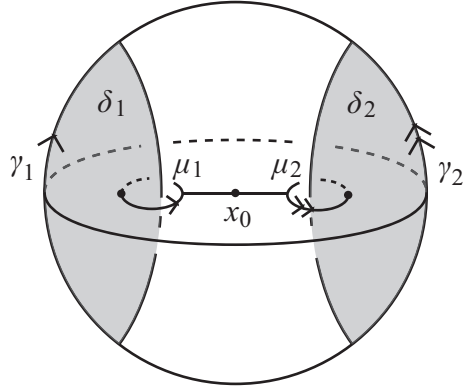}
\end{center}
\caption{
\label{fig.generator}
$\pi_1(B^3-t(\infty), x_0)=F(a,b)$,
where $a$ and $b$ are represented by $\mu_1$ and $\mu_2$, respectively.
}
\end{figure}

To find the upper presentation of $G(K(r))$ explicitly,
let $a$ and $b$, respectively, be the elements of
$\pi_1(B^3-t(\infty), x_0)$
represented by the oriented loops $\mu_1$ and $\mu_2$
based on $x_0$ as illustrated in Figure ~\ref{fig.generator}.
Then $\{a,b\}$ forms the meridian pair of
$\pi_1(B^3-t(\infty))$, which is identified with
the free group $F(a,b)$.
Note that
$\mu_i$ intersects the disk, $\delta_i$, in $B^3$
bounded by a component of $t(\infty)$ and
the essential arc, $\gamma_i$, on
$\partial(B^3,t(\infty))=\Conways$ of slope $1/0$,
in Figure ~\ref{fig.generator}.
Obtain a word $u_r$ in
$\{a, b\}$ by reading the intersection of the
(suitably oriented) loop $\alpha_r$
with $\gamma_1\cup \gamma_2$,
where a positive intersection with $\gamma_1$ (resp. $\gamma_2$)
corresponds to $a$ (resp. $b$).
Then the cyclic word $(u_r)$
represents the free homotopy class of $\alpha_r$
(see Section ~\ref{sequences} for the precise definition of a cyclic word).
It then follows that
\[
\begin{aligned}
G(K(r))&=\pi_1(S^3-K(r))\cong\pi_1(B^3-t(\infty))/\llangle \alpha_r\rrangle \\
&\cong F(a, b)/ \llangle u_r \rrangle
\cong \langle a, b \, | \, u_r \rangle.
\end{aligned}
\]
If $r \neq \infty$,
then $\alpha_r$ intersects $\gamma_1$ and $\gamma_2$ alternately,
and hence $a$ and $b$ appear in $(u_r)$ alternately.
It is known that there is a nice formula to find $u_r$ as follows
(see \cite[Proposition ~1]{Riley}).

\begin{lemma} \label{presentation}
Let $p$ and $q$ be relatively prime
positive
integers
such that $p \ge 1$.
For $1 \le i \le p-1$, let
\[\epsilon_i = (-1)^{\lfloor iq/p \rfloor},\]
where $\lfloor x \rfloor$ is the greatest integer not exceeding $x$.
\begin{enumerate}[\indent \rm (1)]
\item If $p$ is odd, then
\[u_{q/p}=a\hat{u}_{q/p}b^{(-1)^q}\hat{u}_{q/p}^{-1},\]
where
$\hat{u}_{q/p} = b^{\epsilon_1} a^{\epsilon_2} \cdots b^{\epsilon_{p-2}} a^{\epsilon_{p-1}}$.
\item If $p$ is even, then
\[u_{q/p}=a\hat{u}_{q/p}a^{-1}\hat{u}_{q/p}^{-1},\]
where
$\hat{u}_{q/p} = b^{\epsilon_1} a^{\epsilon_2} \cdots a^{\epsilon_{p-2}} b^{\epsilon_{p-1}}$.
\end{enumerate}
\end{lemma}

\begin{remark} \label{epsilon}
{\rm
(1)
The word $\hat{u}_{q/p}$ is obtained from the open line-segment
of slope $q/p$ extending from $(0, 0)$ to $(p, q)$ by ``reading'' its intersection
with the vertical lattice lines (see Figure ~\ref{line}).
The open line-segment cuts the vertical lattice line $x = i$
at the point $P_i$
with height $iq/p$.
Note that $\lfloor iq/p \rfloor$ is the height of the
integer lattice point just beneath $P_i$. Each time the line passes through
another horizontal lattice line, the signs of the $\epsilon_i$'s change.
Similarly, the word $u_{q/p}$ can be read from the closed line-segment
which is obtained by slightly shifting the closed line-segment
of slope $q/p$ joining $(0,0)$ with $(2p, 2q)$ to the upper-left direction
(cf. Proof of Lemma ~\ref{u-word}).

(2) For $r=0/1$ and $r=1/0$, we have $u_{0/1}=ab$ and $u_{1/0}=1$.
}
\end{remark}

\begin{figure}[h]
\begin{center}
\includegraphics{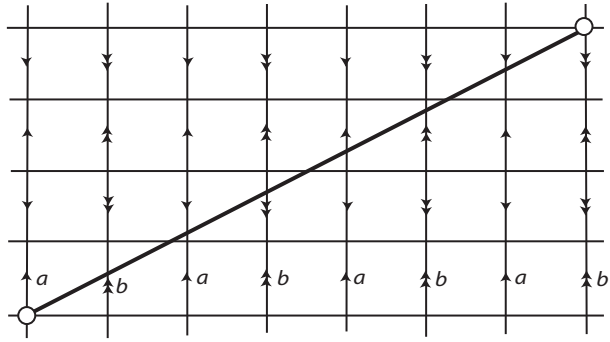}
\end{center}
\caption{\label{line}
The line of slope $4/7$ gives $\hat{u}_{4/7}=ba^{-1}b^{-1}aba^{-1}$, so the
free homotopy class of $\alpha_{4/7}$ is represented by
the cyclic word
$(u_{4/7})=(a\hat{u}_{4/7}b\hat{u}_{4/7}^{-1})=(aba^{-1}b^{-1}aba^{-1}bab^{-1}a^{-1}bab^{-1})$.
Since the inverse image of $\gamma_1$ (resp. $\gamma_2$)
in $\RR^2$ is the union of
the single arrowed (resp. double arrowed)
vertical edges,
a positive intersection with
a single arrowed (resp. double arrowed) edge
corresponds to $a$ (resp. $b$).
}
\end{figure}

In the remainder of this section,
we prove Main Theorem ~\ref{main_theorem_2}
by assuming Main Theorem ~\ref{main_theorem}.
To this end we prepare the following lemma.

\begin{lemma}
\label{some_automorphisms}
{\rm (1)} Let $\varphi$ be the automorphism of
the free group $\pi_1(B^3-t(\infty))=\langle a, b \rangle$
which sends the generating pair $(a,b)$ to
$(a^{-1},b^{-1})$, $(b,a)$ or $(b^{-1},a^{-1})$.
Then $\varphi(u_s)$ is conjugate to
$u_s$ or $u_s^{-1}$ for any $s\in\QQQ$.

{\rm (2)} Let $\varphi$ be the automorphism of
the free group $\pi_1(B^3-t(\infty))=\langle a, b \rangle$
which sends the generating pair $(a,b)$ to
$(a,b^{-1})$, $(a^{-1},b)$,
$(b^{-1},a)$ or $(b,a^{-1})$.
Then $\varphi(u_s)$ is conjugate to
$u_{s+1}$ or $u_{s+1}^{-1}$ for any $s\in\QQQ$.
\end{lemma}

\begin{proof}
(1) Observe that $(B^3,t(\infty))$ admits
a natural $(\ZZ/2\ZZ)^2$-action,
whose generators induce the automorphisms of $\pi_1(B^3-t(\infty))$
sending $(a,b)$ to $(a^{-1},b^{-1})$ and $(b,a)$, respectively.
Moreover, the action preserves the isotopy class of
the (unoriented) loop $\alpha_s$ for every $s\in\QQQ$.
Since any automorphism $\varphi$ satisfying the assumption
is induced by an element of the $(\ZZ/2\ZZ)^2$-action,
we obtain the desired result.

(2) Let $\varphi$ be an automorphism of $\pi_1(B^3-t(\infty))$
satisfying the assumption.
Then it is a composition of an automorphism in (1) and
the automorphism, $\psi$, sending $(a,b)$ to $(a,b^{-1})$.
Observe that $\psi$ is induced by the half-Dehn twist
along the meridian disk of $(B^3,t(\infty))$
and that the half-Denn twist maps $\alpha_s$ to $\alpha_{s+1}$.
Hence we see $\psi(u_s)=u_{s+1}$.
This, together with (1), implies the desired result.
\end{proof}

\begin{proof}[Proof of Main Theorem ~\ref{main_theorem_2}
assuming Main Theorem ~\ref{main_theorem}]

The if part is essentially equivalent to
\cite[Theorem ~1.1]{Ohtsuki-Riley-Sakuma}
and is proved as follows.
If $s$ belongs to the $\RGPP{r}$-orbit of $r$ or $\infty$,
then Main Theorem ~\ref{main_theorem} implies that
$u_s=1$ in $G(K(r))=\langle a, b \, | \, u_r \rangle$.
Thus there is an epimorphism from
$G(K(s))=\langle a, b \, | \, u_s \rangle$ to
$G(K(r))=\langle a, b \, | \, u_r \rangle$
which sends the upper-meridian-pair $(a,b)$ of $G(K(s))$ to
the upper-meridian-pair $(a,b)$ of $G(K(r))$.
To prove the remaining case,
note that there is a homeomorphism
$g:(S^3,K(s))\to (S^3,K(s+1))$
preserving the upper/lower tangles,
such that the restriction of $g$ to
$(B^3,t(\infty))$ is a half-Dehn twist.
Thus $g$ induces an isomorphism
from
$G(K(s))=\langle a, b \, | \, u_s \rangle$ to
$G(K(s+1))=\langle a, b \, | \, u_{s+1} \rangle$
which sends the upper-meridian-pair $(a,b)$ of $G(K(s))$ to
the upper-meridian-pair $(a,b^{-1})$ of $G(K(s+1))$.
So, if $s+1$ belongs to the $\RGPP{r}$-orbit of $r$ or $\infty$,
then we have an epimorphism
$G(K(s))\cong G(K(s+1))\to G(K(r))$
sending $(a,b)$ to $(a,b^{-1})$.

Next, we prove the only if part.
Suppose that there is an upper-meridian-pair preserving
epimorphism $f$ from
$G(K(s))=\langle a, b \, | \, u_s \rangle$ to
$G(K(r))=\langle a, b \, | \, u_r \rangle$.
Then $f$ lifts to an automorphism $\varphi$ of the free
group $\pi_1(B^3-t(\infty))=\langle a,b\rangle$
satisfying the condition in Lemma ~\ref{some_automorphisms},
modulo post composition of an inner-automorphism.
Thus $\varphi(u_s)$ is conjugate to
$u_s$, $u_s^{-1}$, $u_{s+1}$ or $u_{s+1}^{-1}$
by Lemma ~\ref{some_automorphisms}.
Since $\varphi$ is a lift of the homomorphism $f$,
$u_s$ or $u_{s+1}$ represents the trivial element of $G(K(r))$,
accordingly.
Hence, by Main Theorem ~\ref{main_theorem},
we see that $s$ or $s+1$
belongs to the $\RGPP{r}$-orbit of $r$ or $\infty$, accordingly.
\end{proof}

\section{Sequences associated with 2-bridge links}
\label{sequences}

In this section, we define two sequences $S(r)$ and $T(r)$ of slope $r$
and two cyclic sequences $CS(r)$ and $CT(r)$ of slope $r$
all of which arise from the single relator $u_r$ of the presentation
$G(K(r))=\langle a, b \, | \, u_r \rangle$ given in Section ~\ref{group_presentation},
and observe several important properties of these sequences,
so that we
can adopt small cancellation theory in the succeeding sections.

We first fix some definitions and notation.
Let $X$ be a set.
By a {\it word} in $X$, we mean a finite sequence
$x_1^{\epsilon_1}x_2^{\epsilon_2}\cdots x_n^{\epsilon_n}$
where $x_i\in X$ and $\epsilon_i=\pm1$.
Here we call $x_i^{\epsilon_i}$ the {\it $i$-th letter} of the word.
For two words $u, v$ in
$X$, by
$u \equiv v$ we denote the {\it visual equality} of $u$ and
$v$, meaning that if $u=x_1^{\epsilon_1} \cdots x_n^{\epsilon_n}$
and $v=y_1^{\delta_1} \cdots y_m^{\delta_m}$ ($x_i, y_j \in X$; $\epsilon_i, \delta_j=\pm 1$),
then $n=m$ and $x_i=y_i$ and $\epsilon_i=\delta_i$ for each $i=1, \dots, n$.
For example, two words $x_1x_2x_2^{-1}x_3$ and $x_1x_3$ ($x_i \in X$) are {\it not} visually equal,
though
they
are equal as elements of the free group with basis $X$.
The length of a word $v$ is denoted by $|v|$.
A word $v$ in
$X$ is said to be {\it reduced} if $v$ does not contain $xx^{-1}$ or $x^{-1}x$ for any $x \in X$.
A word is called {\it cyclically reduced} if all its cyclic permutations are reduced.
A {\it cyclic word} is defined to be the set of all cyclic permutations of a
cyclically reduced word. By $(v)$ we denote the cyclic word associated with a
cyclically reduced word $v$.
Also by $(u) \equiv (v)$ we mean the {\it visual equality} of two cyclic words
$(u)$ and $(v)$. In fact, $(u) \equiv (v)$ if and only if $v$ is visually a cyclic shift
of $u$.

\begin{definition}
{\rm (1) Let $v$ be a reduced word in
$\{a,b\}$. Decompose $v$ into
\[
v \equiv v_1 v_2 \cdots v_t,
\]
where, for each $i=1, \dots, t-1$, all letters in $v_i$ have positive (resp. negative) exponents,
and all letters in $v_{i+1}$ have negative (resp. positive) exponents.
Then the sequence of positive integers
$S(v):=(|v_1|, |v_2|, \dots, |v_t|)$ is called the {\it $S$-sequence of $v$}.

(2) Let $(v)$ be a cyclic word in
$\{a, b\}$. Decompose $(v)$ into
\[
(v) \equiv (v_1 v_2 \cdots v_t),
\]
where all letters in $v_i$ have positive (resp. negative) exponents,
and all letters in $v_{i+1}$ have negative (resp. positive) exponents (taking
subindices modulo $t$). Then the {\it cyclic} sequence of positive integers
$CS(v):=\lp |v_1|, |v_2|, \dots, |v_t| \rp$ is called
the {\it cyclic $S$-sequence of $(v)$}.
Here the double parentheses denote that the sequence is considered modulo
cyclic permutations.

(3) A reduced word $v$ in $\{a,b\}$ is said to be {\it alternating}
if $a^{\pm 1}$ and $b^{\pm 1}$ appear in $v$ alternately,
i.e., neither $a^{\pm2}$ nor $b^{\pm2}$ appears in $v$.
A cyclic word $(v)$ is said to be {\it alternating}
if all cyclic permutations of $v$ are alternating.
In the latter case, we also say that $v$ is {\it cyclically alternating}.
}
\end{definition}

The following proposition is obvious from the definition.

\begin{proposition} \label{meaning_of_S-sequence}
{\rm (1)}
An alternating word
in $\{a,b\}$ is completely determined by the
initial letter
and the associated $S$-sequence.
\medskip

{\rm (2)} Let $v$ be a cyclically reduced word in
$\{a,b\}$ of length $\ge 2$.
Then the $S$-sequence $S(v)$ represents
the cyclic $S$-sequence $CS(v)$ of $(v)$
if and only if the initial exponent of $v$
is different from the terminal exponent of $v$.
\end{proposition}

\begin{definition}
\label{def4.1(3)}
{\rm
For a rational number $r$ with $0<r\le 1$,
let $u_r$ be the word in $\{a,b\}$ defined in Lemma ~\ref{presentation}.
Then the symbol $S(r)$ (resp. $CS(r)$) denotes the
$S$-sequence $S(u_r)$ of $u_r$
(resp. cyclic $S$-sequence $CS(u_r)$ of $(u_r)$), which is called
the {\it S-sequence of slope $r$}
(resp. the {\it cyclic S-sequence of slope $r$}).}
\end{definition}

We shall first state Propositions
~\ref{properties0}, \ref{properties}, \ref{induction1} and \ref{sequence}
below concerning the sequences defined in the above, and then prove the propositions
in the remainder of this section.
Propositions ~\ref{induction1} and \ref{sequence} play crucial roles
in the proof of Main Theorem ~\ref{main_theorem}.
Though we need those propositions only for
the sequences $S(r)$ and $CS(r)$ with $0< r\le 1$,
we need to extend the definitions of $S(r)$ and $CS(r)$ to
an arbitrary positive rational number $r$
(Definition ~\ref{def4.1(3)_for_r>1}),
in order to prove these propositions.
Thus Propositions
~\ref{properties0}, \ref{properties}, \ref{induction1} and \ref{sequence}
below should be regarded as propositions for
every positive rational number $r$.

Throughout the remainder of this section,
$r=q/p$ denotes a positive rational number,
where $p$ and $q$ are relatively prime positive integers.
Then $r$ has
a continued fraction expansion
\begin{center}
\begin{picture}(230,70)
\put(0,48){$\displaystyle{
r=q/p=
\cfrac{1}{m_1+
\cfrac{1}{ \raisebox{-5pt}[0pt][0pt]{$m_2 \, + \, $}
\raisebox{-10pt}[0pt][0pt]{$\, \ddots \ $}
\raisebox{-12pt}[0pt][0pt]{$+ \, \cfrac{1}{m_k}$}
}} \
=:[m_1,m_2, \dots,m_k] ,}$}
\end{picture}
\end{center}
where $k \ge 1$, $m_1\in \mathbb{Z}_+\cup \{0\}$,
$(m_2, \dots, m_k) \in (\mathbb{Z}_+)^{k-1}$ and
$m_k \ge 2$ unless $k=1$.
Note that $m_1\ge 1$ if $0<r\le 1$,
whereas $m_1=0$ if $r>1$.

\begin{proposition}
\label{properties0}
For the positive rational number $r=q/p$,
the
sequence $S(r)$ has length $2q$,
and it represents the cyclic sequence $CS(r)$.
Moreover the cyclic sequence $CS(r)$ is invariant by
the half-rotation; that is,
if $s_j(r)$ denotes the $j$-th term of $S(r)$
($1\le j\le 2q$), then
$s_j(r)=s_{q+j}(r)$ for every integer $j$ ($1\le j\le q$).
\end{proposition}

\begin{proposition}
\label{properties}
For the positive rational number $r=[m_1,m_2, \dots,m_k]$,
putting
$m=m_1$, we have the following.
\begin{enumerate}[\indent \rm (1)]
\item Suppose $k=1$, i.e., $r=1/m$.
Then $S(r)=(m,m)$.

\item Suppose $k\ge 2$. Then each term of $S(r)$ is either $m$ or $m+1$,
and $S(r)$ begins with $m+1$ and ends with $m$.
Moreover, the following hold.

\begin{enumerate}[\rm (a)]
\item If $m_2=1$, then no two consecutive terms of $S(r)$ can be $(m, m)$,
so there is a sequence
of positive integers $(t_1,t_2,\dots,t_s)$ such that
\[
S(r)=(t_1\langle m+1\rangle, m, t_2\langle m+1\rangle, m, \dots,
t_s\langle m+1\rangle, m).
\]
Here, the symbol ``$t_i\langle m+1\rangle$'' represents $t_i$ successive $m+1$'s.

\item If $m_2 \ge 2$, then no two consecutive terms of $S(r)$ can be $(m+1, m+1)$,
so there is a sequence
of positive integers $(t_1,t_2,\dots,t_s)$ such that
\[
S(r)=(m+1, t_1\langle m\rangle, m+1, t_2\langle m\rangle,
\dots,m+1, t_s\langle m\rangle).
\]
Here, the symbol ``$t_i\langle m\rangle$'' represents $t_i$ successive $m$'s.
\end{enumerate}
\end{enumerate}
\end{proposition}

\begin{remark}
{\rm In \cite{hirasawa-murasugi},
Hirasawa and Murasugi defined,
as one of the key notions of their paper, the sequence of signs for a pair $(p, q)$,
which actually gives rise to our $S$-sequence $S(q/p)$ of slope $q/p$. They also
observed several properties for the sequence of signs for $(p, q)$,  which are
very similar to the properties of $S(q/p)$ stated in Proposition ~\ref{properties}.}
\end{remark}

\begin{definition}
\label{def_T(r)}
{\rm
If $k\ge 2$, the symbol $T(r)$ denotes the sequence
$(t_1,t_2,\dots,t_s)$ in
Proposition ~\ref{properties},
which is called the
{\it $T$-sequence of slope $r$}.
The symbol $CT(r)$ denotes the cyclic
sequence represented by $T(r)$, which is called the
{\it cyclic $T$-sequence of slope $r$}. }
\end{definition}

\begin{example}
\label{cyclic_sequence}
{\rm (1) Let $r={10}/{37}=[3,1,2,3]$.
By Lemma ~\ref{presentation}, we
see that the $S$-sequence of $\hat{u}_r$ is
\[
S(\hat{u}_r)=(3, 4, 4, 3, 4, 4, 3, 4, 4, 3).
\]
By the formula for $u_r$ in Lemma ~\ref{presentation},
this implies
\[
S(r)=S(u_r) =
(\underbrace{4,4,4}_3,3,\underbrace{4, 4}_2, 3, \underbrace{4, 4}_2, 3,
\underbrace{4,4,4}_3,3, \underbrace{4, 4}_2, 3, \underbrace{4, 4}_2, 3).
\]
So $T(r)=(3, 2, 2, 3, 2, 2)$ and $CT(r) = \lp 3, 2, 2, 3, 2, 2 \rp$.

(2) Let $r={8}/{35}=[4,2,1,2]$.
Again by Lemma ~\ref{presentation},
we obtain that the $S$-sequence of $\hat{u}_r$ is
\[
S(\hat{u}_r)=(4, 4, 5, 4, 4, 5, 4, 4).
\]
By the formula for $u_r$ in Lemma ~\ref{presentation},
this implies
\[
S(r)=S(u_r)=
(5, \underbrace{4}_1, 5,
\underbrace{4, 4}_2, 5, \underbrace{4, 4}_2, 5, \underbrace{4}_1, 5,
\underbrace{4, 4}_2, 5, \underbrace{4, 4}_2).
\]
So $T(r) = (1, 2, 2, 1, 2, 2)$ and $CT(r) = \lp 1, 2, 2, 1, 2, 2 \rp$.
}
\end{example}

\begin{proposition}
\label{induction1}
For the rational number $r=[m_1, m_2, \dots, m_k]$,
let $r'$ be the rational number defined as
\[
r'=
\begin{cases}
[m_3, \dots, m_k] & \text{if $m_2=1$};\\
[m_2-1, m_3, \dots, m_k] & \text{if $m_2 \ge 2$}.
\end{cases}
\]
Then we have
\[
T(r)=
\begin{cases}
S(r')  & \text{if $m_2=1$}; \\
\overleftarrow{S}(r') & \text{if $m_2 \ge 2$},
\end{cases}
\]
where $\overleftarrow{S}(r')$ denotes the sequence obtained from
$S(r')$ reversing its order.
\end{proposition}

\begin{proposition}
\label{sequence}
For the positive rational number $r=[m_1,m_2, \dots,m_k]$,
putting $m=m_1$,
the sequence $S(r)$
has a decomposition $(S_1, S_2, S_1, S_2)$ which satisfies the following.
\begin{enumerate} [\indent \rm (1)]
\item Each $S_i$ is symmetric,
i.e., the sequence obtained from $S_i$ by reversing the order is
equal to $S_i$. (Here, $S_1$ is empty if $k=1$.)
\item Each $S_i$ occurs only twice in
the cyclic sequence $CS(r)$.
\item $S_1$ begins and ends with $m+1$.
\item $S_2$ begins and ends with $m$.
\end{enumerate}
\end{proposition}

\begin{corollary}
\label{induction2}
$CS(r)$ is symmetric,
i.e., the cyclic sequence obtained from $CS(r)$
by reversing its cyclic order is equivalent to $CS(r)$
(as a cyclic sequence).
In particular,
in Proposition ~\ref{induction1}, we actually have
\[
CT(r)=CS(r').
\]
\end{corollary}

\begin{example}
{\rm (1) Let $r={10}/{37}=[3,1,2,3]$.
Recall from Example ~\ref{cyclic_sequence} that
\[
S(r) =(4,4,4, 3, 4, 4, 3, 4, 4, 3, 4,4,4,3,4, 4, 3, 4, 4, 3).
\]
Putting $S_1=(4,4,4)$ and $S_2=(3, 4, 4, 3, 4, 4, 3)$, we have
\[
S(r)= (S_1, S_2, S_1, S_2),
\]
where $S_1$ and $S_2$ satisfy all the assertions in Proposition ~\ref{sequence}.

(2) Let $r={8}/{35}=[4,2,1,2]$. Recall also from
Example ~\ref{cyclic_sequence} that
\[
S(r) = (5, 4, 5, 4, 4, 5, 4, 4, 5, 4, 5, 4, 4, 5, 4, 4).
\]
Putting $S_1=(5, 4, 5)$ and $S_2=(4, 4, 5, 4, 4)$, we also have
\[
S(r) = (S_1, S_2, S_1, S_2),
\]
where $S_1$ and $S_2$ satisfy all the assertions in Proposition ~\ref{sequence}.
}
\end{example}

The remainder of this section is devoted to the proof of the propositions.
We first prepare a few symbols.
For a real number $t$,
let $\lfloor t \rfloor$ be the greatest integer not exceeding $t$,
$\lfloor t \rfloor_*$ the greatest integer smaller than $t$,
and $\lceil t \rceil^*$ be the smallest integer greater than $t$.
Then, $\lfloor t\rfloor_*=\lfloor t \rfloor < \lceil t \rceil^*$ for a non-integral real number $t$,
whereas $n-1=\lfloor n \rfloor_*<\lfloor n \rfloor < \lceil n \rceil^*=n+1$ for an integer $n$.
We also note that
$\lfloor t+n \rfloor_* =\lfloor t \rfloor_*+n$ and $\lceil t+n \rceil^* =\lceil t \rceil^*+n$
for every $t\in\RR$ and $n\in\ZZ$.
By using this symbol, we have the following formula for the relator
$u_r$ in the group presentation of $G(K(r))$ given in Section ~\ref{group_presentation}.

\begin{lemma}
\label{u-word}
For the positive rational number $r=q/p$,
the
word $u_r$ is given by the following formula:
\[
u_r=
a^{\varepsilon_1} b^{\varepsilon_2} \cdots a^{\varepsilon_{2p-1}}
b^{\varepsilon_{2p}},
\]
where $\varepsilon_i=(-1)^{\lceil (i-1)q/p \rceil^*-1}$.
In particular, $u_r$ is alternating and cyclically reduced.
\end{lemma}

To prove Lemma ~\ref{u-word},
let $L(r)$ be the line in $\RR^2$ of slope $r$
passing through the origin, and let $L^+(r)$ be the line
obtained by translating $L(r)$ by the vector $(0,\eta)$
for sufficiently small positive real number $\eta$.
Then $L^+(r)$ lies in $\RR^2-\ZZ^2$ and projects to the simple loop
$\alpha_r$. Pick a base point, $z$, from the intersection of $L^+(r)$
with the second quadrant, and consider the sub-line-segment of $L^+(r)$
bounded by $z$ and $z+(2p,2q)$.
Then it forms a fundamental domain of the covering
$L^+(r)\to \alpha_r$, and the word $u_r$ is obtained by reading the
intersection of the line-segment with the vertical lattice lines.
To be precise, for each integer $0\le i\le 2p-1$,
let $P_i^+$ be the intersection of the line-segment with the
vertical lattice line $x=i$.
We define the {\it letter} at $P_i^+$ to be $a$ or $b$
according as $P_i^+$ lies on a vertical edge
with a single arrow
or double arrow in Figure ~\ref{line},
namely according as $i$ is even or odd.
We define the {\it sign} of $P_i^+$ to be $+1$ or $-1$
according as the corresponding arrow is upward or downward.
Then the letter and the sign of $P_i^+$, respectively,
give the letter and the exponent of
the $(i+1)$-th term of the word $u_r$
for each $0\le i\le 2p-1$.
To describe the sign of $P_i^+$,
note that the $y$-coordinate of $P_i^+$
is equal to $iq/p+\eta$, where $\eta$ is a sufficiently small
positive real.
Thus it is contained in the open interval
$(\lceil iq/p \rceil^*-1, \lceil iq/p \rceil^*)$.
Thus the corresponding arrow is
upward or downward
according as $\lceil iq/p \rceil^*-1$ is even or odd.
Hence the sign of $P_i^+$
is equal to $(-1)^{\lceil iq/p \rceil^*-1}$.
This means that the exponent, $\varepsilon_i$, of the $i$-th term of $u_r$ is
$(-1)^{\lceil (i-1)q/p \rceil^*-1}$.
Thus we obtain the first assertion of Lemma ~\ref{u-word}.
The second assertion is a direct consequence of the first assertion.

\begin{remark}
\label{two-epsiolons}
{\rm
For $1\le i\le p-1$,
the sign $\epsilon_i = (-1)^{\lfloor iq/p \rfloor}$ in
Lemma ~\ref{presentation} is of course equal to
the sign $\varepsilon_{i+1}=(-1)^{\lceil iq/p \rceil^*-1}$
in Lemma ~\ref{u-word}
}
\end{remark}

\begin{lemma}
\label{j-term}
If $0< r \le 1$, then
the sequence $S(r)$ has length $2q$, and its
$j$-th term $s_j(r)$ is given by the following formula
($1\le j\le 2q$):
\begin{align*}
s_j(r)
&=
\#\{i\in \{0,1,\dots,2q-1\}\ | \
P_i^+\in \RR\times (j-1,j)\}\\
&=
\#\{i\in \{0,1,\dots,2q-1\}\ | \
\lceil iq/p \rceil^*=j\}\\
&=
\lfloor jp/q \rfloor_*- \lfloor (j-1)p/q \rfloor_*,
\end{align*}
where $\#$ denotes the number of elements of the set.
\end{lemma}

\begin{proof}
Suppose $0< r\le 1$.
Then, for each integer $j$ with $1\le j\le 2q$,
the horizontal strip $\RR\times (j-1,j)$
contains some $P_i^+$, namely,
the right hand side of the first identity is a positive integer.
By this fact and
by the above geometric description of $u_r$
and the definition of $S(r)=S(u_r)$,
we see that $S(r)$ has length $2q$ and that
$s_j(r)$ is equal to the number of the points $P_i^+$'s
which are contained in the horizontal strip $\RR\times (j-1,j)$.
So we obtain the first identity.
As noted in the preceding
argument, the condition
$P_i^+\in \RR\times (j-1,j)$ is equivalent to the condition
$j-1<iq/p+\eta<j$,
where $\eta$ is a sufficiently small
positive real.
This condition is equivalent to the condition that $\lceil iq/p \rceil^*=j$.
Thus we obtain the second identity of the lemma.
To show the last identity, note that the
above
condition is
equivalent to the condition
\[
(j-1)p/q-\eta' < i < jp/q-\eta' \quad
\mbox{for a sufficiently small
positive real $\eta'$.}
\]
This in turn is equivalent to the following condition:
\[
\lfloor (j-1)p/q \rfloor_*< i \le \lfloor jp/q \rfloor_*.
\]
Hence we have
$s_j(r)= \lfloor jp/q \rfloor_*- \lfloor (j-1)p/q \rfloor_*$,
completing the proof of Lemma ~\ref{j-term}.
\end{proof}

The above argument
also
shows that
the three numbers on the right hand side of the identity
in the above lemma are equal even if $r>1$.
Thus the following definition makes sense.

\begin{definition}
\label{def4.1(3)_for_r>1}
{\rm
We extend the definition of
$S(r)$, $CS(r)$, $T(r)$ and $CT(r)$
to an arbitrary positive rational number $r$
by the formula in the above definition.
Namely, for a positive rational number $r=q/p$,
the {\it $S$-sequence of slope $r$}, $S(r)$,
is defined by
\[
S(r)=(s_1(r),s_2(r),\dots, s_{2q}(r)),
\]
where
\begin{align*}
s_j(r)
&=
\#\{i\in \{0,1,\dots,2q-1\}\ | \
P_i^+\in \RR\times (j-1,j)\}\\
&=
\#\{i\in \{0,1,\dots,2q-1\}\ | \
\lceil iq/p \rceil^*=j\}\\
&=
\lfloor jp/q \rfloor_*- \lfloor (j-1)p/q \rfloor_*.
\end{align*}
The {\it cyclic $S$-sequence}, $CS(r)$,
the {\it $T$-sequence}, $T(r)$, and
the {\it cyclic $T$-sequence}, $CT(r)$, of slope $r$
are defined from the above $S(r)$ as in Definitions ~\ref{def4.1(3)}
and \ref{def_T(r)}.
}
\end{definition}

\begin{remark}
\label{remark:S(r)=S(u_r)}
{\rm
Though the word $u_r$ for $r>1$ is already defined
and given by Lemma ~\ref{u-word},
the sequence $S(u_r)$ is different from the sequence $S(r)$.
In fact, $S(u_r)$ consists of only positive integers,
whereas $S(r)$ may contain $0$.
}
\end{remark}

\begin{proof}[Proof of Proposition ~\ref{properties0}]
By Lemma ~\ref{j-term} and Definition ~\ref{def4.1(3)_for_r>1},
$S(r)$ has length $2q$.
Since $u_r$ begins with $a$ and ends with $b^{-1}$ (see Lemma ~\ref{u-word}),
it follows that the sequence $S(r)$ represents the cyclic sequence $CS(r)$
(cf. Proposition ~\ref{meaning_of_S-sequence}(2)).
The symmetry $s_{q+j}(r) =s_j(r)$ is proved as follows:
\begin{align*}
s_{q+j}(r) &= \lfloor (q+j)p/q \rfloor_* - \lfloor (q+j-1)p/q \rfloor_*\\
&=\lfloor p+(jp/q) \rfloor_* - \lfloor p+(j-1)p/q \rfloor_*\\
&=\lfloor jp/q \rfloor_*- \lfloor (j-1)p/q \rfloor_*\\
&=s_j(r).
\end{align*}
We note that the symmetry also follows from the fact that
the translation of $\RR^2$ by the vector $(p,q)$ preserves the
line $L^+(r)$ and maps the horizontal strip bounded by lattice lines
to another such strip.
\end{proof}

For the positive rational number
$r=q/p=[m_1,m_2, \dots,m_k]$,
let $c$ be the non-negative integer defined by
\[
p=m_1q+c.
\]
If $k=1$ then $c=0$, and if $k\ge 2$ then
$0<c<q$.

\begin{lemma}
\label{continued-fraction}
We have the following continued fraction expansions:
\begin{align*}
q/c &=[0,m_2,m_3,\dots,m_k],\\
c/q &=[m_2,m_3,\dots,m_k],\\
(q-c)/c &=[m_3,\dots,m_k] &\text{if $m_2=1$,}\qquad\qquad\qquad\\
c/(q-c) &=[m_2-1,m_3,\dots,m_k],
\\
q/(q-c) &=[0,1,m_2-1,m_3,\dots,m_k].
\end{align*}
\end{lemma}

\begin{proof}
Since $q/p=[m_1,m_2,\dots,m_k]$, we have
$p/q=m_1+[m_2,\dots,m_k]$.
So,
$c/q=(p-m_1q)/q=[m_2,\dots,m_k]$
and hence
$q/c=[0,m_2,m_3,\dots,m_k]$.
Since $q/c=m_2+[m_3,\dots,m_k]$,
we have
$(q-c)/c=(m_2-1)+[m_3,\dots,m_k]$.
So, if $m_2=1$, we have $(q-c)/c=[m_3,\dots,m_k]$.
It also implies that
$c/(q-c) =[m_2-1,m_3,\dots,m_k]$.
Thus
$q/(q-c)=1+c/(q-c)=1+[m_2-1,m_3,\dots,m_k]$,
and hence
$(q-c)/q=[1,m_2-1,m_3,\dots,m_k]$
and
$q/(q-c)=[0,1,m_2-1,m_3,\dots,m_k]$.
\end{proof}

\begin{lemma}
\label{shifting S-formula}
Assume $k\ge 2$ and put $m=m_1$. Then
$S(q/p)=S(q/c)+(m,\dots, m)$.
\end{lemma}

\begin{proof}
Note that $q/c>1$.
By Lemma ~\ref{j-term} and
Definition ~\ref{def4.1(3)_for_r>1},
both $S(q/p)$ and $S(q/c)$ have length $2q$.
Moreover, their components are related as follows:
\begin{align*}
s_j(r)
&=
\lfloor jp/q \rfloor_*- \lfloor (j-1)p/q \rfloor_*\\
&=
\lfloor j(mq+c)/q \rfloor_*- \lfloor (j-1)(mq+c)/q \rfloor_*\\
&=
(jm+ \lfloor jc/q \rfloor_*)-((j-1)m+ \lfloor (j-1)c/q \rfloor_*)\\
&=
m+ \lfloor jc/q \rfloor_*- \lfloor (j-1)c/q \rfloor_*\\
&=
m+s_j(q/c).
\end{align*}
This completes the proof of Lemma ~\ref{shifting S-formula}.
\end{proof}

\begin{lemma}\label{properties01}
Suppose $k \ge 2$.
Then, for the rational number
\[q/c=[0,m_2,m_3,\dots,m_k],\]
the conclusion of Proposition ~\ref{properties} holds.
Namely, each term of $S(q/c)$ is either $0$ or $1$,
and $S(q/c)$ begins with $1$ and ends with $0$.
Moreover, if $m_2=1$, no two consecutive terms of $S(q/c)$ can be $(0, 0)$,
whereas if $m_2 \ge 2$, no two consecutive terms of $S(q/c)$ can be $(1, 1)$.
\end{lemma}

\begin{proof}
By Definition ~\ref{def4.1(3)_for_r>1},
\[s_j(q/c)=\lfloor jc/q \rfloor_*- \lfloor (j-1)c/q \rfloor_*.\]
Since $jc/q-(j-1)c/q=c/q$ is a positive real number less than $1$,
$s_j(q/c)$ is $0$ or $1$.
Moreover $s_1(q/c)= \lfloor c/q \rfloor_*- \lfloor 0 \rfloor_*=0-(-1)=1$ and
$s_{2q}(q/c)= \lfloor 2c \rfloor_*- \lfloor (2q-1)c/q \rfloor_*=(2c-1)-(2c-1)=0$.
Thus $S(q/c)$ begins with $1$ and ends with $0$.

Note that if $m_2=1$ then $1<q-c<c$ and hence $q<2c$,
whereas if $m_2\ge 2$ then $0<c<q-c$ and hence $q>2c$.
On the other hand,
Definition ~\ref{def4.1(3)_for_r>1} implies
\[
s_{j+1}(q/c)+s_j(q/c)=
\lfloor (j+1)c/q \rfloor_*- \lfloor (j-1)c/q \rfloor_*.
\]
Since $(j+1)c/q-(j-1)c/q=2c/q$ is greater than $1$ or less than $1$
according as $m_2=1$ or $m_2\ge 2$,
we see that $\lfloor (j+1)c/q \rfloor_*- \lfloor (j-1)c/q \rfloor_*$ is at least $1$ or at most $1$, accordingly.
In the first case, it is impossible for both
$s_{j+1}(q/c)$ and $s_j(q/c)$ to be $0$,
whereas in the second case,
it is impossible for both
$s_{j+1}(q/c)$ and $s_j(q/c)$ to be $1$.
This completes the proof.
\end{proof}

\begin{proof}[Proof of Proposition ~\ref{properties}]
If $k=1$, then $r=1/m$ and the assertion is obvious.
If $k\ge 2$, then the assertion is a direct consequence of Lemmas ~\ref{shifting
S-formula} and \ref{properties01}.
\end{proof}

In order to prove Proposition ~\ref{induction1}, we prepare the following lemma.

\begin{lemma}
\label{properties02}
Suppose $k\ge 2$ and $m_2=1$.
Assume that
$S((q-c)/c)=(t_1,\dots, t_{2(q-c)})$.
Then
\[
S(q/c)=(t_1\langle 1\rangle,0, t_2\langle 1\rangle,0,\dots,
t_{2(q-c)}\langle 1\rangle,0).
\]
In particular, $T(q/c)=S((q-c)/c)$.
\end{lemma}

\begin{proof}
Since the first term $s_1((q-c)/c)$ of $S((q-c)/c)$
is equal to $t_1$, we see, by
Definition ~\ref{def4.1(3)_for_r>1} that
$\lceil i(q-c)/c \rceil^*=1$ for every integer $i$ such that $0\le i \le t_1-1$.
This together with the condition $s_2((q-c)/c)=t_2$ implies that
$\lceil i(q-c)/c \rceil^*=2$ for every integer $i$ such that  $t_1\le i \le t_1+t_2-1$.
Similarly, for each integer $\ell$
($1\le \ell \le 2(q-c)$),
we have $\lceil i(q-c)/c \rceil^*=\ell$ for every integer $i$ such that
$\sum_{h=1}^{\ell-1}t_h\le i \le \sum_{h=1}^{\ell}t_h-1$.
Since $\lceil i(q-c)/c \rceil^*=\lceil iq/c \rceil^*-i$,
this implies
$\lceil iq/c \rceil^*=\ell+i$
for every integer $i$ such that
$\sum_{h=1}^{\ell-1}t_h\le i \le \sum_{h=1}^{\ell}t_h-1$.
By Definition ~\ref{def4.1(3)_for_r>1}, this implies that
$s_j(q/c)=1$ if and only if $j=\ell+i$ for some integer $\ell$
($1\le \ell \le 2(q-c)$)
and some integer $i$ ($\sum_{h=1}^{\ell-1}t_h\le i \le \sum_{h=1}^{\ell}t_h-1$).
In other words,
$s_j(q/c)=0$ if and only if $j=\ell+\sum_{h=1}^{\ell}t_h$ for some
integer $\ell$
($1\le \ell \le 2(q-c)$).
Hence,\[
S(q/c)=(t_1\langle 1\rangle,0, t_2\langle 1\rangle,0,\dots,
t_{2(q-c)}\langle 1\rangle,0).
\]
\end{proof}

\begin{proof}[Proof of Proposition ~\ref{induction1} for the case $m_2=1$]
Suppose $m_2=1$. Then $r'=[m_3,\cdots, m_k]=(q-c)/c$ by Lemma ~\ref{continued-fraction}.
Thus, by Lemma ~\ref{properties02},
$T(q/c)=S((q-c)/c)=S(r')$.
On the other hand, Lemma ~\ref{shifting S-formula}
implies $T(q/p)=T(q/c)$. Hence we have $T(r)=T(q/c)=S(r')$.
\end{proof}

In order to prove the remaining case of Proposition ~\ref{induction1},
we need the following lemma.

\begin{lemma}
\label{exchange1}
$S(q/c)=
\overleftarrow{S}(q/(q-c))_{0 \leftrightarrow 1}$,
where $\overleftarrow{S}(q/(q-c))_{0 \leftrightarrow 1}$ is obtained from
$S(q/(q-c))$ by reading backwards and by replacing $0$ and $1$.
\end{lemma}

\begin{proof}
By Lemma ~\ref{continued-fraction},
we have
\begin{align*}
q/c &=[0,m_2,m_3,\dots,m_k]\\
q/(q-c) &=[0,1,m_2-1,m_3,\dots,m_k].
\end{align*}
Thus
$S(q/c)$ and $S(q/(q-c))$ consists of $0$ and $1$, by Lemma ~\ref{properties01}.
On the other hand, for each $1\le i\le 2q$,
we have the following identities by
Definition ~\ref{def4.1(3)_for_r>1}:

\begin{align*}
s_i(q/(q-c))
&=
\lfloor i(q-c)/q \rfloor_*-\lfloor (i-1)(q-c)/q \rfloor_*\\
&=
(i+\lfloor -ic/q \rfloor_*)-((i-1)-\lfloor -(i-1)c/q \rfloor_*)\\
&=
1+\lfloor -ic/q \rfloor_*-\lfloor (1-i)c/q \rfloor_*\\
s_{q+1-i}(q/c)
&=
\lfloor (q+1-i)c/q \rfloor_*- \lfloor (q-i)c/q \rfloor_*\\
&=
(1+\lfloor (1-i)c/q \rfloor_*)-(1+ \lfloor -ic/q \rfloor_*)\\
&=
\lfloor (1-i)c/q \rfloor_*- \lfloor -ic/q \rfloor_*.
\end{align*}

Hence, for each $1\le i\le 2q$,
\[
s_i(q/(q-c))+s_{q+1-i}(q/c)=1.
\]
This implies the desired result.
\end{proof}

\begin{corollary}
\label{exchange2}
If $m_2\ge 2$, then
$T(q/c)=\overleftarrow{T}(q/(q-c))$.
\end{corollary}

\begin{proof}
Since $q/c =[0,m_2,m_3,\dots,m_k]$ and since $m_2\ge 2$,
the sequence $T(q/c)$ records the successive occurrences of $0$'s in
$S(q/c)$. On the other hand, since $q/(q-c) =[0,1,m_2-1,m_3,\dots,m_k]$,
$T(q/(q-c))$ records the successive occurrences of $1$'s in
$S(q/(q-c))$. Hence Lemma ~\ref{exchange1} implies the desired result.
\end{proof}

\begin{proof}[Proof of Proposition ~\ref{induction1}
for the case $m_2\ge 2$]
Suppose $m_2\ge 2$.
Then $T(q/c)=\overleftarrow{T}(q/(q-c))$
by Corollary ~\ref{exchange2}.
Note that the $m_2$ for $q/(q-c)$ is equal to $1$,
and the $r'$ for $q/(q-c)$ is equal to $[m_2-1,m_3,\dots,m_k]$,
which is equal to the $r'$ for
the original $r=q/p$.
Hence, it follows from Proposition ~\ref{induction1} for the case $m_2=1$
that $T(q/(q-c))=S(r')$.
Thus we have $T(q/c)=\overleftarrow{T}(q/(q-c))=\overleftarrow{S}(r')$.
Since $T(q/p)=T(q/c)$
by Lemma ~\ref{shifting S-formula},
we obtain the desired identity,
$T(r)=\overleftarrow{S}(r')$.
\end{proof}

\begin{proof}[Proof of Proposition ~\ref{sequence}]
The proof proceeds by induction on $k \ge 1$.
If $k=1$,
$S(r)=(m, m)$.
So putting $S_1$ to be the empty sequence and $S_2=(m)$,
the assertion clearly holds. Now let $k \ge
2$, and let $r'$ be
the rational number defined as in
Proposition ~\ref{induction1}. We consider four cases separately.

\medskip
\noindent {\bf Case 1.} {\it $m_2=1$ and $k=3$}.
\medskip

In this case, $r'=[m_3]$. Thus
$S(r')=(T_1, T_2, T_1, T_2)$,
where $T_1=\emptyset$ and $T_2=(m_3)$.
Put
\[
S_1 =(m_3\langle m+1 \rangle), \quad \text{and} \quad
S_2 =(m).
\]
Since $T(r)=S(r')=(m_3,m_3)$
by Proposition ~\ref{induction1},
we see $S(r) =(S_1, S_2, S_1, S_2)$
by the definition of $T(r)$.
Obviously, $S_1$ and $S_2$ satisfy the desired conditions.

\medskip
\noindent {\bf Case 2.} {\it $m_2\ge 2$ and $k=2$}.
\medskip

In this case, $r'=[m_2-1]$. Thus
$S(r')=(T_1, T_2, T_1, T_2)$,
where $T_1=\emptyset$ and $T_2=(m_2-1)$.
Put
\[
S_1 =(m+1), \quad \text{and} \quad
S_2 =((m_2-1)\langle m \rangle).
\]
Since $T(r)=\overleftarrow{S}(r')= (m_2-1, m_2-1)$ by
Proposition ~\ref{induction1},
we see $S(r) =(S_1, S_2, S_1, S_2)$
by the definition of $T(r)$.
Obviously, $S_1$ and $S_2$ satisfy the desired conditions.

\medskip
\noindent {\bf Case 3.} {\it $m_2=1$ and $k\ge 4$}.
\medskip

In this case, $r'=[m_3, \dots, m_k]$.
By the inductive hypothesis,
\[S(r')=(T_1, T_2, T_1, T_2),\]
where $T_1$ and $T_2$ are symmetric subsequences of
$S(r')$
such that each $T_i$ occurs only twice in $CS(r')$,
$T_1$ begins and ends with $m_3+1$, and such that $T_2$ begins and ends
with $m_3$. Write
\[T_1=(t_1, \dots, t_{s_1}) \quad \text{\rm and} \quad
T_2=(t_{s_1+1}, \dots, t_{s_2}),\] and put
\[
\begin{aligned}
S_1
&=(t_1 \langle m+1 \rangle,  m,
t_2 \langle m+1 \rangle,
\dots, t_{s_1-1}\langle m+1 \rangle, m,
t_{s_1}\langle m+1 \rangle); \\
S_2 &=(m, t_{s_1+1}\langle m+1\rangle,m, \dots,
m, t_{s_2}\langle m+1\rangle, m).
\end{aligned}
\]
Since $T(r)=S(r')$ by
Proposition ~\ref{induction1},
we see $S(r) =(S_1, S_2, S_1, S_2)$
by the definition of $T(r)$.
Since $T_1$ and $T_2$ are symmetric by the inductive hypothesis,
we see that $S_1$ and $S_2$ are symmetric subsequences of
$S(r)$
such that $S_1$ begins and ends with $m+1$, and $S_2$ begins and
ends with $m$.

It remains to show that each $S_i$ occurs
only twice in $CS(r)$. Recall that
$S_1$ begins and ends with $m_3+1$ consecutive $m+1$'s, and that the maximum number of consecutive occurrences
of $m+1$ in $\lp S_1, S_2, S_1, S_2 \rp$ is $m_3+1$
(apply Proposition ~\ref{properties} to $T(r)=S(r')$
and use the definition of $T(r)$).
So, if $S_1$ occurred more than twice in
$\lp S_1, S_2, S_1, S_2 \rp$, $T_1$
also would occur more than twice in $\lp T_1, T_2, T_1, T_2 \rp$, a
contradiction. On the other hand, recall that $m$'s are isolated in $CS(r)$,
and that $S_2$ begins and ends with $m$.
So if $S_2$ occurred more than twice in $\lp S_1, S_2, S_1, S_2 \rp$, $T_2$
also would occur more than twice in $\lp T_1, T_2, T_1, T_2 \rp$, a
contradiction.
\medskip

\noindent {\bf Case 4.} {\it $m_2 \ge 2$ and $k\ge 3$}.
\medskip

In this case, $r'=[m_2-1, m_3, \dots, m_k]$. By the inductive hypothesis,
\[S(r')=(T_1, T_2, T_1, T_2),\]
where $T_1$ and $T_2$ are symmetric subsequences of $CS(r')$
such that each $T_i$ occurs only twice in $CS(r')$, $T_1$
begins and ends with $m_2$, and such that $T_2$ begins and
ends with $m_2-1$. Write
\[T_1=(t_1, \dots, t_{s_1}) \quad \text{\rm and} \quad
T_2=(t_{s_1+1}, \dots, t_{s_2}) ,\]
and put
\[
\begin{aligned}
S_1 &=
(m+1, t_{s_1+1}\langle m\rangle, m+1,
\dots, m+1, t_{s_2}\langle m\rangle, m+1);\\
S_2 &=
(t_1\langle m\rangle, m+1,t_2\langle m\rangle, \dots,
t_{s_1-1}\langle m\rangle, m+1, t_{s_1}\langle m\rangle).
\end{aligned}
\]
Since $T(r)=\overleftarrow{S}(r')$ by
Proposition ~\ref{induction1},
we see $S(r) =(S_1, S_2, S_1, S_2)$
by the definition of $T(r)$ and by using
the fact that $T_1$ and $T_2$ are symmetric.
By using the inductive hypothesis,
we see that $S_1$ and $S_2$ are
symmetric subsequences of $S(r)$
such that $S_1$ begins and ends
with $m+1$, and $S_2$ begins and ends with $m$.
Furthermore, arguing as in Case ~3, we can show that each $S_i$ occurs only twice
in $CS(r)$.
To show the assertion for $S_2$, we use the fact that $S_2$ begins and ends with $m_2$ consecutive $m$'s, and that
the maximum number of consecutive occurrences of $m$ in $\lp S_1, S_2, S_1, S_2 \rp$ is $m_2$.
\end{proof}

\section{Small cancellation conditions for 2-bridge link groups}
\label{small_cancellation}

Let $F(X)$ be the free group with basis $X$. A subset $R$ of $F(X)$ is called {\it symmetrized},
if all elements of $R$ are cyclically reduced and,
for each $w \in R$, all cyclic permutations of $w$ and $w^{-1}$ also belong to $R$.

\begin{definition}{\rm Suppose that $R$ is a symmetrized subset of $F(X)$.
A nonempty word $b$ is called a {\it piece} if there exist distinct $w_1, w_2 \in R$
such that $w_1 \equiv bc_1$ and $w_2 \equiv bc_2$.
Small cancellation conditions $C(p)$ and $T(q)$,
where $p$ and $q$ are integers such that $p \ge 2$ and $q \ge 3$,
are defined as follows (see \cite{lyndon_schupp}).
\begin{enumerate}[\indent \rm (1)]
\item Condition $C(p)$: If
$w \in R$
is a product of $n$ pieces, then $n \ge p$.

\item Condition $T(q)$: For
$w_1, \dots, w_n \in R$
with no successive elements
$w_i, w_{i+1}$
an inverse pair $(i$ mod $n)$, if $n < q$, then at least one of the products
$w_1 w_2,\dots,$ $w_{n-1} w_n$, $w_n w_1$
is freely reduced without cancellation.
\end{enumerate}
}
\end{definition}

In this section, we prove the following key theorem.

\begin{theorem} \label{small_cancellation_condition}
Let $r$ be a rational number such that $0 < r< 1$.
Recall the presentation
$\langle a, b \, |\, u_r \rangle$ of $G(K(r))$
given in Section ~\ref{group_presentation}, and
let $R$ be the symmetrized subset of $F(a, b)$ generated
by the single relator $u_r$.
Then $R$ satisfies $C(4)$ and $T(4)$.
\end{theorem}

In the remainder of this section,
$r$ denotes a rational number such that $0 < r< 1$,
and $(S_1, S_2, S_1, S_2)$ denotes the decomposition
of $S(r)=S(u_r)$
given by Proposition ~\ref{sequence}.
We decompose $u_r \equiv v_1v_2v_3v_4$,
where subwords $v_1$ and $v_3$
correspond to $S_1$, and subwords $v_2$ and $v_4$
correspond to $S_2$.
As in Section ~\ref{sequences},
we consider the continued fraction expansion
$r=[m_1, m_2, \dots, m_k]$,
where $k \ge 1$, $(m_1, m_2, \dots, m_k) \in (\mathbb{Z}_+)^k$ and
$m_k \ge 2$ unless $k=1$.
It should be noted that if $k=1$ then
both $v_1$ and $v_3$ are empty words.

We begin with the following lemma.

\begin{lemma}
\label{initial}
Let $w$ be an arbitrary cyclic permutation
of the single relator $u_r$ of the group presentation of $G(K(r))$.
Then the set
\[
\{ \text{the initial letter of  $w' \, | \, (w') \equiv (u_r^{\pm 1})$ and $S(w')=S(w)$} \}
\]
equals $\{a, a^{-1}, b, b^{-1} \}$.
\end{lemma}

\begin{proof}
We first prove the lemma when $w \equiv u_r (\equiv v_1v_2v_3v_4)$.
Consider the cyclic permutation $w_1 := v_3v_4v_1v_2$ of $w$
and the cyclic permutations
$w_2 := v_1^{-1}v_4^{-1}v_3^{-1}v_2^{-1}$
and $w_3 := v_3^{-1}v_2^{-1}v_1^{-1}v_4^{-1}$ of $u_r^{-1}$.
Then by Proposition ~\ref{sequence},
$w_1$, $w_2$ and $w_3$ share the same $S$-sequence with $w$.
We show that the initial letters of $w$, $w_1$, $w_2$ and $w_3$
are all distinct.
By Lemma ~\ref{presentation},
$w \equiv u_r$ has the initial letter $a$,
and $w_1$ has initial letter $a^{-1}$, $b$ or $b^{-1}$.
Thus $w$ and $w_1$ have different initial letters.
This also implies that $w_2$ and $w_3$ have different initial letters,
as follows.
Suppose $w_2$ and $w_3$ share the same initial letter.
Then, since $S(w_2)=S(w_3)$, we have $w_2 \equiv w_3$
by Proposition ~\ref{meaning_of_S-sequence}(1).
However, this implies $v_1 \equiv v_3$ and $v_2 \equiv v_4$,
and hence $w \equiv w_1$, a contradiction.
Next, we show that the initial letters of $w_2$ and $w_3$
are different from those of $w$ and $w_1$.
Suppose to the contrary that this is not the case.
Then, since these four words have the same $S$-sequences,
it follows from Proposition ~\ref{meaning_of_S-sequence}(1) that
$w_2$ or $w_3$ is equal to $w$ or $w_1$.
This implies that $u_r^{-1}$ is a cyclic permutation of $u_r$.
However, this is impossible by the following claim,
and this completes the proof of the lemma when $w \equiv u_r$.

\medskip
\noindent{\bf Claim.}
{\it $u_r^{-1}$ cannot be a cyclic permutation of $u_r$.}

\begin{proof}[Proof of Claim]
If $u_r^{-1}$ were a cyclic permutation of $u_r$,
then there would be decompositions such as $u_r \equiv z_1z_2$ and
$u_r^{-1} \equiv z_2z_1$.
Since $u_r^{-1} \equiv (z_1z_2)^{-1} \equiv z_2^{-1}z_1^{-1}$,
we would have $z_i \equiv z_i^{-1}$
yielding that $z_i^2=1$  ($i=1,2$) in the free group $F(a,b)$.
Since $F(a,b)$ is torsion free, we have $z_i=1$ ($i=1,2$)
and hence $u_r=z_1z_2=1$ in $F(a,b)$, a contradiction.
\end{proof}

Now, let $w$ be an arbitrary cyclic permutation of $u_r$.
Let $d$ be an integer such that
$w$ is obtained from $u_r$ by cyclical shift of $d$-digits.
For each $i=1,2,3$, let $\hat{w_i}$ be the word
obtained from the word $w_i$ in the previous paragraph
by cyclic shift of $d$-digits.
Then, since $S(w_i)=S(u_r)$, we have $S(\hat{w_i})=S(w)$
($i=1,2,3$).
This implies that the initial letters of $w$ and $\hat{w_i}$ ($i=1,2,3$)
are all distinct.
Because, otherwise, $w$ and  $\hat{w_i}$ ($i=1,2,3$)
are not all distinct by Proposition ~\ref{meaning_of_S-sequence}(1),
and hence $u_r$ and $w_i$ ($i=1,2,3$) are not all distinct,
a contradiction.
Moreover, $\hat{w_i}$ ($i=1,2,3$) are cyclic permutations of $w^{\pm}$,
because
$w_i$ ($i=1,2,3$) are cyclic permutations of $u_r^{\pm 1}$
and $w$ is a cyclic permutation of $u_r$.
Hence we obtain the desired result.
\end{proof}

\begin{lemma}
\label{maximal_piece}
For the relator $u_r \equiv v_1 v_2 v_3 v_4$
with $r=[m_1, m_2, \dots, m_k]$,
the following hold.
\begin{enumerate}[\indent \rm (1)]
\item If $k=1$, then the following hold.

\begin{enumerate}[\rm (a)]
\item No piece can contain $v_2$ or $v_4$.

\item No piece is of the form
$v_{2e} v_{4b}$ or $v_{4e} v_{2b}$,
where $v_{ib}$ and $v_{ie}$ are nonempty initial and terminal subwords of $v_i$, respectively.

\item Every subword of the form $v_{2b}$, $v_{2e}$, $v_{4b}$, or $v_{4e}$ is a piece,
where $v_{ib}$ and $v_{ie}$ are nonempty initial and terminal subwords of $v_i$ with $|v_{ib}|, |v_{ie}| \le |v_i|-1$, respectively.
\end{enumerate}

\item If $k \ge 2$, then the following hold.

\begin{enumerate}[\rm (a)]
\item No piece can contain $v_1$ or $v_3$.

\item No piece is of the form
$v_{1e} v_2 v_{3b}$ or $v_{3e} v_4 v_{1b}$,
where $v_{ib}$ and $v_{ie}$ are nonempty initial and terminal subwords of $v_i$, respectively.

\item Every subword of the form $v_{1e} v_2$, $v_2 v_{3b}$, $v_{3e} v_4$, or $v_4 v_{1b}$ is a piece,
where $v_{ib}$ and $v_{ie}$ are nonempty initial and terminal subwords of $v_i$ with $|v_{ib}|, |v_{ie}| \le |v_i|-1$, respectively.
\end{enumerate}
\end{enumerate}
\end{lemma}

\begin{proof}
(1a) \&  (1b) \& (1c)
The proofs are analogous to the proofs of (2a) \&  (2b) \& (2c) below.

(2a) Suppose to the contrary that
there are two distinct cyclic permutations $w_1$ and $w_2$ of $u_r$
or $u_r^{-1}$ such that $w_1$ and $w_2$ have the same beginning
subword $y$, where either $y \equiv v_1$ or $y \equiv v_3$.
Since the cyclic sequence $CS(r)=CS(u_r)$ is symmetric by
Corollary ~\ref{induction2},
the cyclic sequence $CS(u_r^{-1})$ is also equal to $CS(r)$.
Thus the two cyclic words $(w_1)$ and $(w_2)$
have the same associated cyclic sequence $\lp S_1, S_2, S_1, S_2 \rp$,
regardless of whether $w_1$ and $w_2$ are cyclic permutations of
$u_r$ or $u_r^{-1}$. Putting $m=m_1$, note that $\lp S_1, S_2, S_1, S_2 \rp$ is a cyclic
sequence consisting of only $m$ and $m+1$, $S_1$ begins and ends
with $m+1$, and the $S$-sequence of $y$ is $S_1$
(see Proposition ~\ref{properties}).
This
implies that, for each $i=1, 2$, the $S$-sequence $S(w_i)$ begins with
$S_1$, and that the cyclic $S$-sequence $CS(w_i)$ is represented
by $S(w_i)$. Furthermore, since $S_1$ appears only
twice in $\lp S_1, S_2, S_1, S_2 \rp$
by Proposition ~\ref{sequence},
we obtain that $w_1$ and $w_2$ must
have the same associated $S$-sequence $(S_1, S_2, S_1, S_2)$, where
the first $S_1$ corresponds to the common beginning subword $y$.
By Proposition ~\ref{meaning_of_S-sequence}(1),
this implies that $w_1 \equiv w_2$, a contradiction.

(2b)
Suppose to the contrary that there is a piece, $z$,
which is of the form, say $v_{1e} v_2 v_{3b}$,
and let $w_1$ and $w_2$ be distinct cyclic permutations of $u_r$
or $u_r^{-1}$ such that $w_1$ and $w_2$ have the same beginning
subword $z\equiv v_{1e} v_2 v_{3b}$,
namely $w_i\equiv zw_i'\equiv v_{1e} v_2 v_{3b}w_i'$ for some
subword $w_i'$ of
$w_i$. By the construction of the product $v_1v_2v_3v_4$,
the last exponent of $v_{1e}$,
which is equal to the last exponent of $v_1$,
is different from the first exponent of $v_2$.
Consider the cyclic permutation
$\hat w_i:=v_2 v_{3b}w_i'v_{1e}$.
By the observation above,
the cyclic sequence $CS(\hat w_i)$ is represented by $S(\hat w_i)$
(cf. Proposition ~\ref{meaning_of_S-sequence}(2)).
Moreover, since $v_{3b}$ is a nonempty reduced subword of $\hat w_i$
whose initial exponent is different from
the terminal exponent of $v_2$,
the sequence $S(\hat w_i)$ starts with $S(v_2)=S_2$.
Since $S_2$ appears in $CS(\hat w_i)=\lp S_1, S_2, S_1, S_2 \rp$
only twice, we see $S(\hat w_i)=(S_2, S_1, S_2, S_1)$.
This implies
$\hat w_1 \equiv \hat w_2$ by Proposition ~\ref{meaning_of_S-sequence}(1)
and hence $w_1 \equiv w_2$, a contradiction.

(2c) Since every nonempty subword of a piece is also a piece, it is enough to
prove the assertion for $v_{1e} v_2$, $v_2 v_{3b}$, $v_{3e} v_4$, or $v_4 v_{1b}$,
where $v_{ib}$ and $v_{ie}$, respectively,
are the initial and the terminal subwords
of $v_i$ with $|v_{ib}|=|v_{ie}| = |v_i|-1$.

We show that $v_{1e} v_2$ and $v_2 v_{3b}$ are pieces.
To this end, we first show that
$v_{1e} v_2$ and $v_2 v_{3b}$ have the same associated $S$-sequence.
Since the terminal exponent of $v_{i}$ and the initial exponent
of $v_{i+1}$ are different,
$S(v_{1e} v_2)=(S(v_{1e}), S(v_2))$ and
$S(v_2 v_{3b})=(S(v_2), S(v_{3b}))$.
On the other hand, we have $v_{1e} v_2=\hat{u}_r$,
because $u_r \equiv a \hat{u}_r x \hat{u}_r^{-1}$
by Lemma ~\ref{presentation}.
Thus we see $S(v_{1e} v_2)=S(\hat{u}_r)$ is symmetric
by the following claim.

\medskip
\noindent {\bf Claim.} {\it The sequence $S(\hat{u}_r)$ is symmetric.}

\begin{proof} [Proof of Claim.]
Recall that the $i$-th exponent of $\hat{u}_r$ is
given by $\epsilon_i = (-1)^{\lfloor iq/p \rfloor}$
(see Lemma ~\ref{presentation}).
So we have:
\[
\epsilon_{p-i}
=(-1)^{\lfloor (p-i)q/p \rfloor}
=(-1)^{q+\lfloor -iq/p \rfloor}
=
\begin{cases}
\epsilon_i & \quad \mbox{if $q$ is even},\\
-\epsilon_i & \quad \mbox{if $q$ is odd}.
\end{cases}
\]
Hence the sequence $(\epsilon_1,\epsilon_2,\dots,\epsilon_{p-1})$
is symmetric or skew-symmetric according as $q$ is even or odd.
This implies that $S(\hat{u}_r)$ is symmetric.
\end{proof}

\noindent
Hence we have
\[
\begin{aligned}
S(v_{1e} v_2)&=\overleftarrow{S}(v_{1e} v_2)=(\overleftarrow{S}(v_2), \overleftarrow{S}(v_{1e})) \\
&=(S(v_2), S(v_{1b}))=(S(v_2), S(v_{3b}))=S(v_2 v_{3b}).
\end{aligned}
\]
Here, the third identity follows from the fact that
$S(v_1)=S_1$ and $S(v_2)=S_2$ are symmetric
and the fourth identity follows from the fact that
$S(v_1)=S_1=S(v_3)$.

Now, let
$w_1 :=v_{1e} v_2 w_1'$ and $w_2:=v_2 v_{3b}w_2'$ be
cyclic permutations of $u_r$.
Note that
the terminal exponent of $v_2$ and
the initial exponent of $w_1'$ are different,
and that
the terminal exponent of $v_{3b}$ and the initial exponent of $w_2'$
are the same.
Here, the latter assertion follows from the fact that
the last component of $S(v_3)=S_1$ is equal to
$m_1$ or $m_1+1$ according as $k=1$ or $k\ge 2$
(see Proposition ~\ref{properties0})
and hence it is at least $2$.
(Recall that $m_1\ge 2$ or $m_1\ge 1$ according as
$k=1$ or $k\ge 2$.)
Hence $S(w_1) \neq S(w_2)$.
By Lemma ~\ref{initial},
there is a cyclic permutation $\hat{w_2}$ of $u_r$ or $u_r^{-1}$ such that
$\hat{w_2}$ has the same initial letter as $w_1$ and such that $S(\hat{w_2})=S(w_2)$.
Then $w_1$ and $\hat{w_2}$ are distinct cyclic permutations of $u_r$ or $u_r^{-1}$,
since $S(\hat{w_2})=S(w_2) \neq S(w_1)$.
Note, however, that $w_1$ and $\hat{w_2}$ have the same beginning subword $v_{1e} v_2$
(cf. Proposition ~\ref{meaning_of_S-sequence}(1)).
This implies that $v_{1e} v_2$ is a piece.
We can also see that $v_2 v_{3b}$ is a piece by a similar argument.
By using the fact that $v_{3e} v_4=\hat{u}_r^{-1}$,
we can show by a similar argument that
$v_{3e} v_4$ and $v_4 v_{1b}$ are also pieces.
\end{proof}

We now introduce the following definition.
\begin{definition}
\label{def:n-piece}
{\rm
For a positive integer $n$,
a nonempty subword $w$ of the cyclic word $(u_r)$
is called a {\it maximal $n$-piece}
if $w$ is a product of $n$ pieces and
if any subword $w'$ of $u_r$
which properly contains $w$ as an {\it initial} subword
is not a product of $n$-pieces.}
\end{definition}

It should be noted that a maximal $1$-piece $w$
may not be a maximal piece,
because there may exist a piece $w'$ which contains
$w$ as a proper terminal subword.
(Here a nonempty subword $w$ of the cyclic word $(u_r)$
is called a {\it maximal piece}
if $w$ is a piece and
if any subword $w'$ of $u_r$
which properly contains $w$ is not a piece.)
However, every maximal piece is a maximal $1$-piece.

\begin{corollary}
\label{cor:n-piece}
For the relator $u_r \equiv v_1 v_2 v_3 v_4$
with $r=[m_1, m_2, \dots, m_k]$,
let $v_{ib}^*$
be the maximal proper initial subword of $v_i$,
i.e., the initial subword of $v_i$
such that $|v_{ib}^*|=|v_i|-1$ ($i=1,2,3,4$).
Then the following hold,
where $v_{ib}$ and $v_{ie}$ are nonempty initial and terminal subwords of $v_i$
with $|v_{ib}|, |v_{ie}| \le |v_i|-1$, respectively.

\begin{enumerate}[\indent \rm (1)]
\item If $k=1$, then the following hold.

\begin{enumerate}[\rm (a)]

\item The following is the list of all
maximal $1$-pieces of $(u_r)$,
arranged in the order of
the position of the initial letter:
\[
v_{2b}^*, v_{2e}, v_{4b}^*, v_{4e}.
\]

\item The following is the list of all
maximal $2$-pieces of $(u_r)$,
arranged in the order of
the position of the initial letter:
\[
v_2, v_{2e} v_{4b}^*, v_4, v_{4e} v_{2b}^*.
\]

\item The following is the list of all
maximal $3$-pieces of $(u_r)$,
arranged in the order of
the position of the initial letter:
\[
v_2 v_{4b}^*, v_{2e} v_4, v_4 v_{2b}^*,
v_{4e} v_2.
\]
\end{enumerate}

\item If $k \ge 2$, then the following hold.

\begin{enumerate}[\rm (a)]

\item The following is the list of all
maximal $1$-pieces of $(u_r)$,
arranged in the order of
the position of the initial letter:
\[
v_{1b}^*, v_{1e} v_2, v_2 v_{3b}^*, v_{2e} v_{3b}^*,
v_{3b}^*, v_{3e} v_4, v_4 v_{1b}^*, v_{4e} v_{1b}^*.
\]

\item The following is the list of all
maximal $2$-pieces of $(u_r)$,
arranged in the order of
the position of the initial letter:
\[
v_1 v_2, v_{1e} v_2 v_{3b}^*,
v_2 v_{3} v_4,
v_{2e} v_3 v_4,
v_3 v_4,  v_{3e} v_4 v_{1b}^*,
v_4 v_{1} v_2,
v_{4e} v_1 v_2.
\]

\item The following is the list of all
maximal $3$-pieces of $(u_r)$,
arranged in the order of
the position of the initial letter:
\[
v_1 v_2 v_{3b}^*, v_{1e} v_2 v_{3} v_4,
v_2 v_{3} v_4 v_{1b}^*,
v_{2e} v_3 v_4 v_{1b}^*,
v_3 v_4 v_{1b}^*, v_{3e} v_4 v_{1} v_2,
v_4 v_1 v_2 v_{3b}^*,
v_{4e} v_1 v_2 v_{3b}^*.
\]
\end{enumerate}
\end{enumerate}
\end{corollary}

\begin{proof}
(1a) \&  (1b) \& (1c)
The proofs are analogous to the proofs of (2a) \&  (2b) \& (2c) below.

(2a) This is a direct consequence of Lemma ~\ref{maximal_piece}.

(2b) This is proved by using the fact that
if $w$ is a maximal $2$-piece,
then it has a unique
decomposition $w=w_1w_2$ into two maximal $1$-pieces
$w_1$ and $w_2$.
To be precise,
if $w_1$ is equal to $v_{1b}^*$
(resp. $v_{1e} v_2$, $v_2 v_{3b}^*$, $v_{2e} v_{3b}^*$),
then $w_2$ is equal to $v_{1e} v_2$
(resp. $v_{3b}^*$, $v_{3e} v_4$, $v_{3e} v_4$),
and hence $w=w_1w_2$ is equal to
$v_1 v_2$ (resp. $v_{1e} v_2 v_{3b}^*$,
$v_2 v_{3} v_4$,
$v_{2e} v_3 v_4$).

(2c) This is proved by using the fact that
if $w$ is a maximal $3$-piece,
which is a proper subword of the cyclic word $(u_r)$,
then it has a unique
decomposition $w=w_1w_2$,
where $w_1$ is a maximal $2$-piece and $w_2$ is a maximal $1$-piece.
To be precise,
if $w_1$ is equal to $v_1 v_2$ (resp. $v_{1e} v_2 v_{3b}^*$,
$v_2 v_{3} v_4$,
$v_{2e} v_3 v_4$),
then $w_2$ is equal to $v_{3b}^*$
(resp. $v_{3e} v_4$, $v_{1b}^*$, $v_{1b}^*$),
and hence $w=w_1w_2$ is equal to
$v_1 v_2 v_{3b}^*$
(resp. $v_{1e} v_2 v_{3} v_4$,
$v_2 v_{3} v_4 v_{1b}^*$,
$v_{2e} v_3 v_4 v_{1b}^*$).
\end{proof}

\begin{proof} [Proof of Theorem ~\ref{small_cancellation_condition}]
By Corollary ~\ref{cor:n-piece},
the cyclic word $(u_r)$ is not a product of $3$ pieces.
This implies that the cyclic word
$(u_r^{-1})$ as well is not a product of $3$ pieces.
Hence $R$ satisfies $C(4)$.
So we show that $R$ satisfies $T(4)$.
To this end, recall that
the cyclic word $(u_r)$ is alternating by
Lemma ~\ref{u-word}.
Now suppose that $R$ does not satisfy $T(4)$.
Then there exist $w_1, w_2, w_3\in R$
such that $w_1w_2$, $w_2w_3$ and $w_3w_1$ are reducible.
Let $x^{\epsilon_1}$ be the terminal letter of $w_1$,
where $x\in\{a,b\}$ and $\epsilon_1=\pm 1$.
Then the initial letter of $w_2$ is $x^{-\epsilon_1}$,
because $w_1w_2$ is reducible.
Since $w_2$ is cyclically alternating, this implies that
the terminal letter of $w_2$ is $y^{\epsilon_2}$
for some $\epsilon_2=\pm1$,
where $y$ is the element of $\{a,b\}$ different from $x$.
Similarly, by using the facts that $w_2w_3$ is reducible
and that $w_3$ is cyclically alternating,
we see that the terminal letter of $w_3$ is $x^{\epsilon_3}$
for some $\epsilon_3=\pm1$.
Since $w_3w_1$ is reducible,
this implies that the initial letter of $w_1$ is $x^{-\epsilon_3}$.
However, this contradicts the fact that
$w_1$ is
cyclically
alternating, because
the terminal letter of $w_1$ was $x^{\epsilon_1}$.
Hence $R$ satisfies $T(4)$.
\end{proof}

\section{Van Kampen diagrams over 2-bridge link groups}
\label{van_Kampen_diagrams}

In this section, we investigate the geometric consequences of
Theorem ~\ref{small_cancellation_condition}.
Let us begin with necessary definitions and notation following \cite{lyndon_schupp}.
A {\it map} $M$ is a finite $2$-dimensional cell complex
embedded in $\RR^2$,
namely a finite collection of vertices ($0$-cells), edges ($1$-cells),
and faces ($2$-cells) in $\RR^2$.
The boundary (frontier) of $M$ in $\RR^2$
is denoted by $\partial M$.
If $D$ is a face of $M$, the boundary of $D$ is denoted by $\partial D$.
An edge may be traversed in either of two directions.
If $v$ is a vertex of
$M$, $d_M(v)$, the {\it degree of $v$}, will
denote the number of oriented edges in $M$ having $v$ as initial vertex.
A vertex $v$ of $M$ is called an {\it interior vertex}
if $v\not\in \partial M$, and an edge $e$ of $M$ is called
an {\it interior edge} if $e\not\subset \partial M$.

\begin{definition}
{\rm A nonempty map $M$ is called a {\it $[p, q]$-map} if the following conditions hold.
\begin{enumerate}[\indent \rm (1)]
\item Every interior vertex of $M$
has degree at least $p$.
\item Every face $D$ of $M$ has at least $q$ edges in $\partial D$.
\end{enumerate}
}
\end{definition}

A {\it path} in $M$ is a sequence of oriented edges $e_1, \dots, e_n$ such that
the initial vertex of $e_{i+1}$ is the terminal vertex of $e_i$ for
every $1 \le i \le n-1$. A {\it cycle} is a closed path, namely
a path $e_1, \dots, e_n$
such that the initial vertex of $e_1$ is the terminal vertex of $e_n$.
If $D$ is a face of $M$, any cycle of minimal length which includes
all the edges of $\partial D$ is called a {\it boundary cycle} of $D$.
If $M$ is connected and simply connected, a {\it boundary cycle} of $M$
is defined to be a cycle of minimal length which contains all the edges of
$\partial M$ going around once
along the boundary of $\mathbb{R}^2 -M$.

\begin{definition}
{\rm Let $R$ be a symmetrized subset of $F(X)$. An {\it $R$-diagram} is
a map $M$ and a function $\phi$ assigning to each oriented edge $e$ of $M$, as a {\it label},
a reduced word $\phi(e)$ in $X$ such that the following hold.
\begin{enumerate}[\indent \rm (1)]
\item If $e$ is an oriented edge of $M$ and $e^{-1}$ is the oppositely oriented edge, then $\phi(e^{-1})=\phi(e)^{-1}$.

\item For any boundary cycle $\delta$ of any face of $M$,
$\phi(\delta)$ is a cyclically reduced word
representing an element of $R$.
(If $\alpha=e_1, \dots, e_n$ is a path in $M$, we define $\phi(\alpha) \equiv \phi(e_1) \cdots \phi(e_n)$.)
\end{enumerate}
In particular, if a group $G$ is presented by $G=\langle X \,|\, R \, \rangle$ with $R$ being symmetrized,
then a connected and simply connected $R$-diagram is called a {\it van Kampen diagram}
over the group presentation $G=\langle X \,|\, R \, \rangle$.
}
\end{definition}

Let $D_1$ and $D_2$ be faces (not necessarily distinct) of $M$
with an edge $e \subseteq \partial D_1 \cap \partial D_2$.
Let $e \delta_1$ and $\delta_2e^{-1}$ be boundary cycles of $D_1$ and $D_2$, respectively.
Let $\phi(\delta_1)=f_1$ and $\phi(\delta_2)=f_2$.
An $R$-diagram $M$ is called {\it reduced} if one never has $f_2=f_1^{-1}$.
It should be noted that if $M$ is reduced
then $\phi(e)$ is a piece for every interior edge $e$ of $M$.
A {\it boundary label of $M$} is defined to be a word $\phi(\alpha)$ in $X$ for $\alpha$ a boundary cycle of $M$.
It is easy to see that any two boundary labels of $M$ are cyclic permutations of each other.

We recall the following lemma which is a well-known
classical result in combinatorial group theory
(see \cite{lyndon_schupp}).

\begin{lemma} [van Kampen]
\label{van_Kampen}
Suppose $G=\langle X \,|\, R \, \rangle$ with $R$ being symmetrized.
Let $v$ be a word in $X$. Then $v=1$ in $G$ if and only if
there exists a reduced van Kampen diagram $M$
over $G=\langle X \,|\, R \, \rangle$
with a boundary label $v$.
\end{lemma}

\begin{convention}
\label{convention}
{\rm Let $R$ be the symmetrized subset of $F(a, b)$
generated by the single relator $u_r$ of the group presentation of $G(K(r))$.
For any reduced $R$-diagram $M$,
we assume that $M$ satisfies the following.
\begin{enumerate}[\indent \rm (1)]
\item Every interior vertex of $M$ has degree at least three.

\item For every edge $e$ of $\partial M$,
the label $\phi(e)$ is a piece.

\item For a path $e_1, \dots, e_n$ in $\partial M$ of length $n\ge 2$
such that the vertex $e_i\cap e_{i+1}$ has degree $2$
for $i=1,2,\dots, n-1$,
$\phi(e_1) \phi(e_2) \cdots\phi(e_n)$ cannot be expressed as a product of less than $n$ pieces.
\end{enumerate}
Indeed, we may assume (1), because if there are two interior edges $e_1$ and $e_2$ meeting in
an interior vertex
of degree two, then we can delete the vertex $v$ and unite $e_1$ and $e_2$ into a single edge $e$
with label $\phi(e)=\phi(e_1)\phi(e_2)$.
To see (2), recall that the assumption that $M$ is reduced
implies that $\phi(e)$ is a piece for every interior edge $e$ of $M$.
On the other hand, since the cyclic word $(u_r)$ can be written as a product of pieces,
we may also assume that $\phi(e)$ is a piece for every edge $e$ in $\partial M$.
Finally, we may assume (3),
because if
$\phi(e_1) \cdots\phi(e_n)$
is expressed as a product of less than $n$ pieces,
then we can change the cellular structure
of the interval
$e_1\cup \cdots \cup e_n$
so that the new cellular structure has fewer vertices
compared with the original one.
}
\end{convention}

For the remainder of this section,
we assume that $r$ is a rational number such that $0 < r< 1$, and let
$\langle a, b \, |\, u_r \rangle$ be the presentation of $G(K(r))$
given in Section ~\ref{group_presentation}.

The following corollary is immediate from Theorem ~\ref{small_cancellation_condition}
and Convention ~\ref{convention}.

\begin{corollary}
\label{small_cancellation_condition_2}
Let $R$ be the symmetrized subset of $F(a, b)$
generated by the single relator $u_r$ of the group presentation of $G(K(r))$.
Then every reduced $R$-diagram is a $[4, 4]$-map.
\end{corollary}

This corollary enables us to apply the Curvature Formula of Lyndon and Schupp
for $[p, q]$-maps satisfying ${1/p}+{1/q}={1/2}$ (see \cite{lyndon_schupp})
to obtain the following theorem,
the proof of which is deferred to the end of this section.

\begin{theorem}
\label{half}
Let $R$ be the symmetrized subset of $F(a, b)$ generated by the single relator $u_r$
of the group presentation of $G(K(r))$. Suppose that $M$ is a reduced van Kampen diagram over
$G(K(r))=\langle a, b \, | \, R \, \rangle$ such that any boundary label of $M$ is cyclically reduced
and alternating.
Then some boundary label of $M$ contains a subword $w$ of $(u_r^{\pm 1})$
such that the $S$-sequence of $w$ is $(S_1, S_2, \ell)$ or $(\ell, S_2, S_1)$
for some positive integer $\ell$,
where $S(r)= (S_1, S_2, S_1, S_2)$
is as in Proposition ~\ref{sequence}.
\end{theorem}

By Lemma ~\ref{van_Kampen}, we obtain the following important corollary
which is the main result of this section.

\begin{corollary}
\label{existence_condition}
Let $s$ be a rational number such that  $0< s\le 1$ and
that $\alpha_s$ is null-homotopic in $S^3-K(r)$.
Then the cyclic $S$-sequence $CS(s)$ contains
$(S_1,S_2)$ or $(S_2,S_1)$ as a subsequence,
where $S(r)= (S_1, S_2, S_1, S_2)$
is as in Proposition ~\ref{sequence}.
\end{corollary}

In the above corollary (and throughout this paper),
we mean by a {\it subsequence}
a subsequence without leap.
Namely a sequence $(a_1,a_2,\dots, a_p)$
is called a {\it subsequence} of a cyclic sequence,
if there is a sequence $(b_1,b_2,\dots, b_n)$
representing the cyclic sequence
such that $p\le n$ and
$a_i=b_i$ for $1\le i\le p$.

\begin{proof}
Let $R$ be the symmetrized subset of $F(a, b)$ generated by $u_r$.
Since $\alpha_s$ is null-homotopic in $S^3-K(r)$,
the cyclic word $(u_s)$ obtained from $\alpha_s$ (as in Section ~3)
represents the trivial element of $G(K(r))=\langle a, b \, | \, R \, \rangle$.
Thus, by Lemma ~\ref{van_Kampen}, there is a reduced van Kampen diagram $M$ over
$G(K(r))=\langle a, b \, | \, R \, \rangle$ with a boundary label $u_s$.
Since $u_s$ is cyclically reduced and the cyclic word $(u_s)$ is
alternating,
Theorem ~\ref{half} implies that
the cyclic word $(u_s)$ contains a subword $w$ such that
the $S$-sequence of $w$ is $(S_1, S_2, \ell)$ or
$(\ell, S_2, S_1)$ for some positive integer $\ell$.
Recall that
$S_1$ begins and ends with $m_1+1$, and $S_2$ begins and ends with $m_1$
(see Proposition ~\ref{sequence}).
Thus $(S_1, S_2, \ell)$ is of the form $(m_1+1,\dots,m_1,\ell)$
and $(\ell, S_2, S_1)$ is of the form $(\ell,m_1,\dots,m_1+1)$.
By Proposition ~\ref{properties}, this yields
that $CS(s)=CS(u_s)$ (cf. Definition ~\ref{def4.1(3)})
consists of $m_1$ and $m_1+1$ and that $CS(s)$ contains
$(S_1, S_2)$ or $(S_2, S_1)$ as a subsequence.
\end{proof}

The remainder of this section is devoted to the proof of Theorem ~\ref{half}.
An {\it extremal disk} of a map $M$ is a submap $J$ of $M$ which is topologically a disk
and which has a boundary cycle $e_1, \dots, e_n$
such that the edges $e_1, \dots, e_n$ occur in order in some boundary cycle of the whole map $M$.
We note that
if $J$ is an extremal disk of $M$, then either $J=M$
or $J$ is connected to the rest of $M$ by a single vertex.

\begin{proof} [Proof of Theorem ~\ref{half}]
By Corollary ~\ref{small_cancellation_condition_2},
a reduced van Kampen diagram $M$ over $G(K(r))=\langle a, b \, | \, R \, \rangle$ is a $[4,4]$-map.
Since a boundary label of $M$ is cyclically reduced, there is no vertex of degree $1$ in $\partial M$.
Moreover, since
any boundary label of $M$ is alternating,
there is no vertex of degree $3$ in $\partial M$. So
every vertex in $\partial M$ must have degree $2$ or at least $4$.

Choose an extremal disk, say $J$, of $M$.

\medskip
\noindent
{\bf Claim.} {\it There are three edges $e_1$, $e_2$ and $e_3$ in $\partial J$ such that $e_1 \cap e_2=\{v_1\}$
and $e_2 \cap e_3=\{v_2\}$, where $d_J(v_i)=2$ for each $i=1, 2$.}

\begin{proof} [Proof of Claim] Clearly $J$ is a connected and simply connected $[4, 4]$-map
having at least one face. By the Curvature Formula of
Lyndon and Schupp (see ~\cite[Corollary ~V.3.4]{lyndon_schupp}),
we have
\[ \tag{\dag} \label{curvature_formula}
\sum_{v \in \, \partial J} (3-d_J(v)) \ge 4.
\]
Putting
\[
A=\{v \in \partial J \, | \, d_J(v)=2\} \quad \text{and} \quad
B=\{v \in \partial J \, | \, d_J(v) \ge 4\},
\]
it is easy to see that $A$ has at least $4$ more elements than $B$ does
in order to satisfy inequality ~(\ref{curvature_formula}).
Since $J$ is an extremal disk of $M$, either $J=M$
or it is connected to the rest of $M$ by a single vertex.
If $J=M$, then every vertex in $\partial J=\partial M$ belongs to either $A$ or $B$.
On the other hand, if $J$ is connected to the rest of $M$ by a single vertex,
say $v_0$, then every vertex in $\partial J$ except $v_0$ belongs to either $A$ or $B$
and $d_J(v_0)=d_M(v_0)-1 \ge 3$ (note that $d_M(v_0) \ge 4$, since $v_0 \in \partial M$).
In either case, we see that there are at least $2$ adjacent vertices,
say $v_1$ and $v_2$, belonging to $A$.
This proves the claim.
\end{proof}

By Claim, there is a face $D$ in $M$ such that $\partial D \cap \partial M$ contains three successive edges
$e_1$, $e_2$ and $e_3$.
By Convention ~\ref{convention}(2)--(3), the product $\phi(e_1) \phi(e_2) \phi(e_3)$
which is a subword of the cyclic word $(u_r^{\pm 1})$ cannot be expressed as a product of less than $3$ pieces.
We may assume
without loss of generality that
$\phi(e_1) \phi(e_2) \phi(e_3)$ is a subword of
the cyclic word $(u_r)$.
We also assume that the length $k$
of the continued fraction $r=[m_1,m_2, \dots,m_k]$
is greater than $1$.
(The proof for the case $k=1$ is analogous to the proof
for the general case $k\ge 2$.)
Let $w_0$
be the maximal $2$-piece which forms a proper initial
subword of $\phi(e_1) \phi(e_2) \phi(e_3)$.
Then $w_0$
is equal to one of the words in
Corollary ~\ref{cor:n-piece}(2b)
If $w_0$
is equal to $v_1 v_2$ or $v_{1e} v_2 v_{3b}^*$, then
$\phi(e_1) \phi(e_2) \phi(e_3)$ contains
a subword $w$
such that the $S$-sequence of
$w$ is $(S_1,S_2,\ell)$ or
$(\ell, S_2, S_1)$ accordingly,
for some positive integer $\ell$.
The remaining
possibilities for $w_0$
can be treated similarly
and we obtain Theorem ~\ref{half}.
\end{proof}

\section{Proof of Main Theorem ~\ref{main_theorem}}
\label{Proof_of_MainTheorem}

In this section, we prove the only if part of Main Theorem ~\ref{main_theorem}.
The if part is Proposition ~\ref{Prop:Knotgroup}
(\cite[Corollary ~4.7]{Ohtsuki-Riley-Sakuma}).
Though the proof of
the main theorem for the trivial knot $K(0/1)$ and the trivial $2$-component link $K(1/0)$ is easy,
we need to treat them separately.
We defer these
to the end of this section, and
we assume, until the final part of this section,
that the slope $r$ of the $2$-bridge link $K(r)$
satisfies the condition $0<r<1$ (cf. Theorem ~\ref{Thm:Schubert})
and that $r=[m_1, m_2, \dots, m_k]$, where $k \ge 1$,
$(m_1, \dots, m_k) \in (\mathbb{Z}_+)^k$, and $m_k \ge 2$.

Recall that the region $R$ bounded by a pair of
Farey edges with an endpoint $\infty$
and a pair of edges with an endpoint $r$
forms a fundamental domain of the action of $\RGPP{r}$ on $\HH^2$
(see Figure ~\ref{fig.fd}).
Let $I_1$ and $I_2$ be the closed intervals in $\RRR$
obtained as the intersection with $\RRR$ of the closure of $R$.
To be precise, $I_1=[0,r_1]$ and $I_2=[r_2,1]$,
where
\begin{align*}
r_1 &=
\begin{cases}
[m_1, m_2, \dots, m_{k-1}] & \mbox{if $k$ is odd,}\\
[m_1, m_2, \dots, m_{k-1}, m_k-1] & \mbox{if $k$ is even,}
\end{cases}\\
r_2 &=
\begin{cases}
[m_1, m_2, \dots, m_{k-1}, m_k-1] & \mbox{if $k$ is odd,}\\
[m_1, m_2, \dots, m_{k-1}] & \mbox{if $k$ is even.}
\end{cases}
\end{align*}
If $r=1/p$ ($p>1$), then
$I_1$
is degenerate to the singleton
$\{0\}$.
And if $r=(p-1)/p$ ($p>1$), then
$I_2$ is degenerate to the singleton
$\{1\}$.
Otherwise, $I_1$ and $I_2$ are
non-degenerate intervals,
and the union $I_1\cup I_2$ forms
a fundamental domain of the action of $\RGPP{r}$
on the domain of discontinuity of $\RGPP{r}$,
the complement in $\RRR$ of the closure
of $\RGPP{r}\{\infty, r\}$.
(In the exceptional case $r=1/p$ (resp. $(p-1)/p$),
the rational number
$0$ (resp. $1$)
lies in the limit set and
$I_2$ (resp. $I_1$)
is a fundamental domain of the action of
$\RGPP{r}$
on the domain of discontinuity.)
This fact together with Proposition ~\ref{Prop:Knotgroup}
implies the following lemma.

\begin{lemma}
\label{Lemma:FundametalDomain}
Suppose $0<r<1$.
Then for any $s\in\QQQ$,
there is a unique rational number
$s_0\in I_1\cup I_2\cup \{\infty, r\}$
such that $s$ is contained in the  $\RGPP{r}$-orbit of $s_0$,
and in particular, $\alpha_s$ is homotopic to $\alpha_{s_0}$ in
$S^3-K(r)$.
\end{lemma}

\begin{proof}
Let $s$ be an element of $\QQQ$.
Pick a point, $z$, in the interior of
a Farey triangle contained in the fundamental domain $R$,
and consider the geodesic, $\ell$, in $\HH^2$ joining $z$ with $s$.
Then $\ell$ intersects only finitely many Farey edges,
and hence it intersects only finitely many $\RGPP{r}$-images
of the four boundary edges of $R$.
This enables us to find an element $\gamma\in \RGPP{r}$
such that
$\gamma(s)\in \overline{R}\cap \RRR =I_1\cup I_2\cup\{\infty,r\}$.
Thus $s$ is contained in the  $\RGPP{r}$-orbit of
$s_0:=\gamma(s)\in I_1\cup I_2\cup\{\infty,r\}$.
The uniqueness of such an element $s_0$
can be seen by looking at
the quotient space of
$\HH^2\cup \Omega(\RGPP{r})$ by $\RGPP{r}$,
where $\Omega(\RGPP{r})$ is the domain of discontinuity
of the action of $\RGPP{r}$ on $\partial \HH^2$.
Finally, Proposition ~\ref{Prop:Knotgroup}
implies that
$\alpha_s$ is homotopic to $\alpha_{s_0}$ in $S^3-K(r)$.
\end{proof}

Thus the only if part of Main Theorem ~\ref{main_theorem}
is equivalent to the following theorem,
except for the the trivial knot $K(0)$ and
the trivial $2$-component link $K(\infty)$.

\begin{theorem}
\label{if_part_theorem}
Suppose $0<r<1$.
Then,
for any $s\in I_1\cup I_2$,
$\alpha_s$ is not null-homotopic in $S^3-K(r)$.
\end{theorem}

The following lemma plays an important role in the proof of
Theorem ~\ref{if_part_theorem}.

\begin{lemma} \label{connection}
Suppose $0<r<1$,
and let $S(r)= (S_1, S_2, S_1, S_2)$
be as in Proposition ~\ref{sequence}.
Suppose that a rational number
$s\in (0,1)$
has a continued fraction expansion
$s=[l_1, \dots, l_t]$, where $t \ge 1$,
$(l_1, \dots, l_t) \in (\mathbb{Z}_+)^t$, and $l_t \ge 2$ unless $t=1$.
If the cyclic $S$-sequence $CS(s)$ contains
$(S_1, S_2)$ or $(S_2, S_1)$ as a subsequence,
then the following hold.

\begin{enumerate} [\indent \rm (1)]
\item $t \ge k$.

\item $l_i=m_i$ for each $i=1, \dots, k-1$.

\item Either $l_k \ge m_k$ or both $l_k=m_k-1$ and $t>k$.
\end{enumerate}
\end{lemma}

\begin{proof} From Proposition ~\ref{sequence}, keep in mind that
$CS(s)$ consists of $l_1$ and $l_1+1$ (here $l_1+1$ appears only if $t \ge 2$).
The proof proceeds by induction on $k \ge 1$.

If $k=1$, that is, $r=[m_1]$, then
$S_1=\emptyset$ and $S_2=(m_1)$.
So, if $CS(s)$ contains $(S_1,S_2)=(S_2,S_1)=(m_1)$
as a subsequence,
then either $l_1 \ge m_1$ or both $l_1=m_1-1$ and $t \ge 2$,
proving the base step.

Now let $k \ge 2$.
Suppose that $CS(s)$ contains
$(S_1, S_2)$ or $(S_2, S_1)$ as a subsequence.
By Proposition ~\ref{properties}, this yields
that $CS(s)$ consists of $m_1$ and $m_1+1$.
This happens only when $t \ge 2$ and $l_1=m_1$.
We consider three cases separately.

\medskip
\noindent {\bf Case 1.} {\it $m_2=1$.}
\medskip

In this case, $k \ge 3$ and,
by Propositions ~\ref{properties} and \ref{induction1},
$(m_1+1, m_1+1)$ appears in both $(S_1, S_2)$ and $(S_2, S_1)$ as a subsequence, so in $CS(s)$ as a subsequence.
Again by Proposition ~\ref{properties}, $l_2=1$ and so $t \ge 3$. Define
\[
r':=[m_3, \dots, m_k] \quad \text{\rm and} \quad s':=[l_3, \dots, l_t].
\]
Let $S(r')= (T_1, T_2, T_1, T_2)$ be the decomposition of $S(r')$
given by Proposition ~\ref{properties},
as in the proof of Proposition ~\ref{sequence}.
By Corollary ~\ref{induction2},
\[
CT(r)=CS(r') \quad \text{\rm and} \quad CT(s)=CS(s'),
\]
so it follows that
$(T_1, T_2)$ or $(T_2, T_1)$ appears in $CS(s')$ as a subsequence,
because
$(S_1, S_2)$ or $(S_2, S_1)$ appears in $CS(s)$ as a subsequence,
by assumption.
Thus the induction completes the case.

\medskip
\noindent {\bf Case 2.} {\it Both $m_2=2$ and $k=2$.}
\medskip

In this case, the assertion always holds, because if $l_2=1$ then we must have $t \ge 3$, otherwise $l_2 \ge 2=m_2$.

\medskip
\noindent {\bf Case 3.} {\it Either $m_2 \ge 3$ or both $m_2=2$ and $k \ge 3$.}
\medskip

In this case, by Proposition ~\ref{properties},
$(m_1, m_1)$ appears in both $(S_1, S_2)$ and $(S_2, S_1)$ as a subsequence,
so in $CS(s)$ as a subsequence. Again by Proposition ~\ref{properties}, $l_2 \ge 2$.
Define
\[
r':=[m_2-1, m_3, \dots, m_k] \quad \text{\rm and} \quad s':=[l_2-1, l_3, \dots, l_t].
\]
Let $S(r')= (T_1, T_2, T_1, T_2)$ be the decomposition of $S(r')$
given by Proposition ~\ref{properties},
as in the proof of Proposition ~\ref{sequence}.
By Corollary ~\ref{induction2},
\[
CT(r)=CS(r') \quad \text{\rm and} \quad CT(s)=CS(s'),
\]
so it follows that
$(T_1, T_2)$ or $(T_2, T_1)$ appears in $CS(s')$ as a subsequence,
because
$(S_1, S_2)$ or $(S_2, S_1)$ appears in $CS(s)$ as a subsequence,
by assumption.
As in Case ~1, the induction completes the case.
\end{proof}

\begin{remark}\rm
\label{remark:connection}
We can easily see that the a rational number $s\in (0,1]$
satisfies the conclusion of Lemma ~\ref{connection}
if and only if $s$ lies in the open interval
$(r_1,r_2)=(0,1]-(I_1\cup I_2)$,
where $r_1$ and $r_2$ are rational numbers
such that $I_1=[0,r_1]$ and $I_2=[r_2,1]$,
introduced in the paragraph preceding Lemma ~\ref{Lemma:FundametalDomain}.
\end{remark}

We are now in a position to prove
Theorem ~\ref{if_part_theorem}, i.e.,
the only if part of Main Theorem ~\ref{main_theorem}.

\begin{proof} [Proof of Theorem ~\ref{if_part_theorem}]
Consider a $2$-bridge link $K(r)$ with $0<r<1$,
and pick a rational number $s$ from $I_1\cup I_2$.
Suppose on the contrary that
$\alpha_s$ is null-homotopic in $S^3-K(r)$, namely
$u_s=1$ in $G(K(r))$.
If $s\in (0,1]$, then
we see by Corollary ~\ref{existence_condition}
that $CS(s)$ contains $(S_1,S_2)$ or $(S_2, S_1)$
as a subsequence.
Hence, we see by Lemma ~\ref{connection}
and Remark ~\ref{remark:connection}
that
$s\in (r_1,r_2)=(0,1]-(I_1\cup I_2)$, a contradiction.
So, the only possibility is $s=0$.
This case can be handled by directly using Theorem ~\ref{half},
which implies
that $u_s$ must contain a subword $w$ of $(u_r^{\pm 1})$
such that the $S$-sequence of $w$ is
$(S_1, S_2, \ell)$ or $(\ell, S_2, S_1)$
for some positive integer $\ell$.
Note that the length of such a subword $w$ is strictly greater than
$p$, half the length of $(u_r^{\pm 1})$,
where $r=q/p$.
Since $0<r<1$, we have $p\ge 2$.
So, the word $u_0=ab$ cannot contain such a subword,
a contradiction.
This completes the proof
of Theorem ~\ref{if_part_theorem}.
\end{proof}

Thus we have proved
Main Theorem ~\ref{main_theorem} except for the case $r=0$ and $r=\infty$.
These exceptional cases are treated as follows.
Suppose $r=\infty$, then $K(\infty)$ is the trivial $2$-component link,
and $G(K(\infty))$ is the free group $F(a,b)$.
On the other hand, for every $s\in \QQ$,
$u_s$ is a non-trivial cyclically reduced word in $\{a,b\}$
and hence it represents a non-trivial element of $G(K(\infty))$.
On the other hand, the $\RGPP{r}$-orbit of
$\{\infty, r\}=\{\infty\}$ is the singleton $\{\infty\}$.
Hence
Main Theorem ~\ref{main_theorem}
holds for this case.
Next, suppose $r=0$.
Then $G(K(0))=\langle a, b \, | \, ab \rangle\cong \ZZ$.
Further, $\RGPP{r}$ is equal to the group
generated by the reflections in the edges of any of $\DD$.
In particular, any Farey triangle
is a fundamental domain for the action of $\RGPP{r}$ on $\HH^2$.
Hence, any $s\in\QQQ$ belongs to the $\RGPP{r}$-orbit of one
and only one of $\{0,1,\infty\}$.
On the other hand,
$u_1=ab^{-1}=a^2 \neq 1$ in $G(K(0))$.
Hence, Main Theorem ~\ref{main_theorem} holds for this case.
This completes the proof of
Main Theorem ~\ref{main_theorem}.

\begin{remark}
{\rm
The assertion
in \cite[Example ~4.2]{Ohtsuki-Riley-Sakuma}
that $\RGPP{1}$ acts transitively on $\QQQ$
is obviously incorrect.
It should be noted that though there is
an upper-meridian-pair preserving epimorphism (actually an isomorphism)
from $G(K(1))=\langle a, b \, | \, ab^{-1} \rangle$ to
$G(K(0))=\langle a, b \, | \, ab \rangle$,
it does not send the pair $(a,b)$ to $(a,b)$.
}
\end{remark}

At the end of this section,
we describe a geometric intuition
behind the proof of
the main theorem.
Note that a slope $s$ belongs to $I_1\cup I_2$
if and only if it does not belong to
$(-\infty, 0)\cup (1,\infty]$ nor $(r_1,r_2)$.
The condition that $s\notin (-\infty, 0)\cup (1,\infty]$,
i.e., $s\in [0,1]$, implies that the word $u_s$ can be read from
a line of slope $r$ in $\RR^2$ \lq\lq effectively''
so that $S(s)=S(u_s)$
(see Remark ~\ref{remark:S(r)=S(u_r)}).
To describe the geometric meaning of the condition $s\notin (r_1,r_2)$,
set $p_i$ and $q_i$ be relatively prime integers such that
$r_i=q_i/p_i$ $(i=1,2)$.
Then $(p,q)=(p_1+p_2,q_1+q_2)$, where $r=q/p$,
and the parallelogram in $\RR^2$
spanned by $(0,0)$, $(p_1,q_1)$, $(p_2,q_2)$ and $(p,q)$
does not contain lattice points in its interior.
If $s\in (r_1,r_2)$, then the ray
(in the first quadrant)
of slope $s$ from the origin
passes through the interior of the parallelogram
and hence the word $u_s$ shares a long common initial subword
with $u_r$.
On the other hand, if $s\notin (r_1,r_2)\cup(-\infty, 0)\cup (1,\infty]$,
then the ray
(in the first quadrant)
of slope $s$ from the origin
is disjoint from the interior of the parallelogram,
and hence, $u_s$ shares only a short initial subword with $u_r$.
This convinces us that the cyclic word $(u_s)$,
for $s\notin (r_1,r_2)\cup(-\infty, 0)\cup (1,\infty]$,
shares only short common subwords with
the cyclic word $(u_r)$.
This is the intuition behind the proof of
the main theorem.

We realized through discussion with Norbert A'Campo
that the decomposition
$S(r)=(S_1, S_2, S_1, S_2)$ in Proposition ~\ref{sequence}
has a natural geometric interpretation in terms of
the above parallelogram.
To describe it, assume $q_1/p_1 < q/p < q_2/p_2$ in the above setting,
and consider the infinite broken line, $B$,
obtained by joining the lattice points
\[
\dots, (0,0), (p_2,q_2), (p,q), (p+p_2, q+q_2), (2p,2q), \dots
\]
which is invariant by the translation $(x,y)\mapsto (x+p,y+q)$.
By slightly modifying $B$ near the lattice points,
we obtain a (topological) line, $B^+$, in $\RR^2-\ZZ^2$,
invariant by the translation,
which is homotopic to the line $L^+(r)$ in the proof of
Lemma ~\ref{u-word}.
Pick a point, $z_0\in B^+$ in the second quadrant,
and consider the sub-path of $B^+$ bounded by
$z_0$ and $z_4:=z_0+(2p,2q)$.
Then the word $u_r$ is also obtained by reading the
intersection of the sub-path with the vertical lattice lines.
Pick a point $z_1\in B^+$ whose $x$-coordinate is
$p_2+\mbox{(small positive number)}$, and set
$z_2:=z_0+(p,q)$ and $z_3:=z_1+(p,q)$.
Let $B^+_i$ be the sub-path of $B^+$
joining $z_{i-1}$ with $z_i$ ($i=1,2,3,4$).
Then we can see that the subword of $u_r$ corresponding to $B^+_i$
is equal to the word $v_i$ in Section ~\ref{small_cancellation},
i.e., $u_r\equiv v_1v_2v_3v_4$,
$S_1=S(v_1)=S(v_3)$ and $S_2=S(v_2)=S(v_4)$.
In particular, $|v_1|=|v_3|=p_2+1$ and $|v_2|=|v_4|=p_1-1$.
We hope to fully describe this on another occasion.

\section{Relation with a question by Minsky}
\label{Further_discussion}

In this section, we describe the relation of
Main Theorem ~\ref{main_theorem} with
the question raised by Minsky in \cite[Question ~5.4]{Gordon}.
Let $M=H_+\cup_S H_-$ be a Heegaard splitting of
a $3$-manifold $M$.
Let $\Gamma_{\pm}:=MCG(H_{\pm})$ be the mapping
class group of $H_{\pm}$, and
let $\Gamma^0_{\pm}$ be the kernel of the map
$MCG(H_{\pm})\to \mathrm{Out}(\pi_1(H_{\pm}))$.
Identify $\Gamma^0_{\pm}$ with a subgroup of $MCG(S)$,
and consider the subgroup $\langle \Gamma^0_+,\Gamma^0_-\rangle$
of $MCG(S)$.
Now let $\Delta_{\pm}$ be the set of (isotopy classes of)
simple loops in $S$ which bound a disk in $H_{\pm}$.
Let $Z$ be the set of
essential
simple loops in $S$
which are null-homotopic in $M$.
Note that $Z$ contains $\Delta_{\pm}$ and invariant
under $\langle \Gamma^0_+,\Gamma^0_-\rangle$.
In particular, the orbit
$\langle \Gamma^0_+,\Gamma^0_-\rangle(\Delta_+\cup\Delta_-)$
is a subset of $Z$.
Then Minsky posed the following question.

\begin{question}
\label{Minsky_question}
When is $Z$ equal to
the orbit
$\langle \Gamma^0_+,\Gamma^0_-\rangle(\Delta_+\cup\Delta_-)$?
\end{question}

The above question makes sense
not only for Heegaard splittings but also
bridge decompositions of knots and links.
Actually, the groups
$\RGP{\infty}$ and $\RGP{r}$
in our setting correspond to the groups
$\Gamma^0_+$ and $\Gamma^0_-$, and hence the group
$\RGPP{r}$ corresponds to
the group
$\langle\Gamma^0_+,\Gamma^0_-\rangle$.
To make this precise, recall the bridge decomposition
$(S^3,K(r))=(B^3,t(\infty))\cup (B^3,t(r))$,
and let $\tilde\Gamma_{+}$ (resp. $\tilde\Gamma_{-}$) be the mapping
class group of the pair $(B^3,t(\infty))$ (resp. $(B^3,t(r))$), and
let $\tilde\Gamma^0_{\pm}$ be the kernel of the natural map
$\tilde\Gamma_{+}\to \mathrm{Out}(\pi_1(B^3-t(\infty)))$
(resp. $\tilde\Gamma_{-}\to \mathrm{Out}(\pi_1(B^3-t(r)))$).
Identify $\tilde\Gamma^0_{\pm}$ with a subgroup of
the mapping class group $MCG(\PConway)$ of the
$4$-times punctured sphere $\PConway$.
Recall that the Farey tessellation $\DD$ is identified with the curve complex of $\PConway$ and there is a natural map
from $MCG(\PConway)$ to the automorphism group $Aut(\DD)$
of $\DD$,
whose kernel is equal to the image of the
$(\ZZ/2\ZZ)^2$-action on $\PConway$, which appeared
in the proof of Lemma ~\ref{some_automorphisms}.
Then the group $\Gamma_{\infty}$ (resp. $\Gamma_r$)
introduced in Section ~\ref{statements}
is identified with the image of
$\tilde\Gamma^0_{+}$ (resp. $\tilde\Gamma^0_{-}$)
by this natural map.
Moreover, the sets $\{\alpha_{\infty}\}$
and $\{\alpha_{r}\}$ correspond to
the sets $\Delta_+$ and $\Delta_-$,
and Main Theorem ~\ref{main_theorem} says that
the set $Z$ of simple loops in $\PConway$
which are null-homotopic in $S^3-K(r)$
is equal to the orbit
$\langle \Gamma_{\infty}, \Gamma_r\rangle(\Delta_+\cup\Delta_-)$.
Thus Main Theorem ~\ref{main_theorem}
may
be regarded as an answer to
the special variation of Question ~\ref{Minsky_question}.

Finally, we note that Main Theorem ~\ref{main_theorem}
is also related to the existence of a possible variation
of McShane's identity for $2$-bridge knots
(see \cite{Sakuma}).
Related topics are studied in
subsequent papers
\cite{lee_sakuma_2, lee_sakuma_3, lee_sakuma_4,
lee_sakuma_5}.
For an overview of this series of works,
please see the research announcement \cite{lee_sakuma_6}.


\bibstyle{plain}

\end{document}